\documentclass[10pt,reqno]{amsart}
\usepackage[margin=0.8in]{geometry}
\usepackage{amsthm, amsmath,amsfonts,amssymb,euscript,hyperref,graphics,color,slashed}
\usepackage{graphicx}
\usepackage{mathrsfs}
\usepackage{comment}
\usepackage{latexsym}
\usepackage[makeroom]{cancel}

\def\inte#1{
\displaystyle\mathop{#1\kern0pt}^\circ }

\let\e=\varepsilon
\let\z=\zeta

\let\f=\frac

\let\p=\psi

\def\virgp{\raise 2pt\hbox{,}}
\def\cdotpv{\raise 2pt\hbox{;}}

\def\eqdefa{\buildrel\hbox{\footnotesize def}\over =}

\def\C{\mathop{\mathbb C\kern 0pt}\nolimits}
\def\DD{\mathop{\mathbb D\kern 0pt}\nolimits}
\def\EE{\mathop{{\mathbb E \kern 0pt}}\nolimits}
\def\K{\mathop{\mathbb K\kern 0pt}\nolimits}
\def\N{\mathop{\mathbb N\kern 0pt}\nolimits}
\def\Q{\mathop{\mathbb Q\kern 0pt}\nolimits}
\def\R{\mathop{\mathbb R\kern 0pt}\nolimits}
\def\SS{\mathop{\mathbb S\kern 0pt}\nolimits}
\def\ZZ{\mathop{\mathbb Z\kern 0pt}\nolimits}
\def\TT{\mathop{\mathbb T\kern 0pt}\nolimits}
\def\P{\mathop{\mathbb P\kern 0pt}\nolimits}

\newcommand{\Z}{{\ZZ}}

\def\na{\nabla}
\def\p{\partial}

\newcommand{\beq}{\begin{equation}}
\newcommand{\eeq}{\end{equation}}
\newcommand{\ben}{\begin{eqnarray}}
\newcommand{\een}{\end{eqnarray}}
\newcommand{\beno}{\begin{eqnarray*}}
\newcommand{\eeno}{\end{eqnarray*}}

\newtheorem{thm}{Theorem}[section]

\newcommand{\vv}[1]{\boldsymbol{#1}}

\def\v{v}

\newtheorem*{Main Theorem}{Main Theorem}
\newtheorem{theorem}{Theorem}[section]
\newtheorem{lemma}[theorem]{Lemma}
\newtheorem{proposition}[theorem]{Proposition}

\newtheorem{remark}[theorem]{Remark}

\numberwithin{equation}{section}

\begin{document}
\title[Well-posedness]{Long time existence for a strongly dispersive Boussinesq system}

\author{Jean-Claude Saut}
\address{Laboratoire de Math\' ematiques, UMR 8628\\
Universit\' e Paris-Saclay, Paris-Sud et CNRS\\ 91405 Orsay, France}
\email{jean-claude.saut@u-psud.fr}

\author[Li XU]{Li Xu}
\address{School of Mathematics and Systems Science, Beihang University\\  100191 Beijing, China}
\email{xuliice@buaa.edu.cn}

\date{September17, 2019}
\maketitle

\vspace{1cm}
 \textit{Abstract}.

 This paper is concerned with the one-dimensional version of a specific member of the (abcd) family of Boussinesq systems having the higher possible dispersion. We will establish two different  long time existence results for the solutions of the Cauchy problem. The first result concerns the system \eqref{Bsq} without a small parameter. If  the initial data  is of order $O(\e)$, we  prove that the existence time scale is of $1/\e^{\f43}$ which improves the  result  $1/\e$ that could be obtained by a "dispersive" method.   The second result is about the system \eqref{Bsqeps} which involves a  small parameter $\epsilon$ in front of the dispersive and nonlinear terms and which is the form obtained when the system is derived from the water wave system in the KdV/Boussinesq regime. If the initial data is of order $O(1)$, we obtain the existence time scale $1/{\epsilon^{\f23}}$ which improves the result $1/\sqrt{\epsilon}$  obtained by a dispersive method.  These results were not included in the previous papers dealing with similar issues because of the presence of zeroes in the phases. The proof involves normal form transformations suitably modified away from the zero set of the phases.\\


\tableofcontents

\setcounter{equation}{0}
\section{Introduction}
\subsection{The general setting}

The four-parameter (abcd) Boussinesq systems for {\it long wavelength, small amplitude} gravity-capillary  surface water waves introduced in \cite{BCL, BCS1} couples  the elevation of the wave $\zeta=\zeta(x,t)$ to a measure of the horizontal velocity $\vv v=\vv v(x,t), x\in \R^N, N=1,2, t\in \R$ and read as follows:
\beq\label{Bsq 1}\left\{\begin{aligned}
&\p_t\z+\na\cdot\vv v+\epsilon\na\cdot(\z\vv v)+\epsilon\bigl(a\na\cdot\Delta\vv v-b\Delta\p_t\z\bigr)=0,\\
&\p_t\vv v+\na\z+\f{\epsilon}{2}\na(|\vv v|^2)+\epsilon\bigl(c\na\Delta\z-d\Delta\p_t\vv v\bigr)=\vv 0.
\end{aligned}\right.\eeq

Here $a, b, c, d$ are modeling parameters which satisfy the
constraint $a+b+c+d=\frac{1}{3}-\tau $  where $\tau\geq 0$ is a measure of surface tension effects, $\tau=0$ for pure gravity waves.


In \eqref{Bsq 1},  the small parameter $\epsilon$ is defined by
$$\epsilon=a/h\sim (h/\lambda)^2,$$
 where $h$ denotes the mean depth of the fluid, $a$ a typical amplitude of the wave and $\lambda$ a typical horizontal wavelength.

It was established in \cite{BCL} that, in suitable Sobolev classes, the error with solutions of the full water waves system and the approximation given by \eqref{Bsq 1} is of order $O(\epsilon^2 t).$ This result is of course useful if one knows that the corresponding solutions of the water wave system in this regime and of the Boussinesq systems exist on time scales of at least $O(1/\epsilon).$ This has been proven in \cite{AL}, see also \cite{La1}, for the water wave systems and in \cite{ Bu, Bu2, MSZ, SX, SWX} for all the locally-well posed Boussinesq systems except the case $b=d=0, a=c>0$ which is in some sense special since the "generic" case  $b=d=0, a, c>0, a\neq c$ is linearly ill-posed.

 \begin{remark}
 The global well-posedness of Boussinesq systems has been only established in a few one-dimensional cases, including the case $a=c=b=0, d>0$ that can be viewed as a dispersive perturbation of the hyperbolic Saint-Venant (shallow water) system, see \cite{A, Sc}, and the  Hamiltonian cases $b=d>0, a\leq 0, c<0$, see \cite{BCS2}. We also refer to \cite{KMPP, KM} for scattering results in the energy space for those Hamiltonian cases when $b=d>0.$
\end{remark}

\vspace{0.3cm}
Recall that the linearization of \eqref{Bsq 1} around the null solution is well-posed (see \cite{BCS1}) provided that
\beq\label{Bsq 2}
a\leq0,\quad c\leq0,\quad b\geq0,\quad d\geq0,
\eeq
\beq\label{Bsq 3}
\text{or}\quad a=c>0,\quad b\geq0,\quad d\geq0.
\eeq

Actually the linear well-posedness occurs when the non zero eigenvalues of the linearization of \eqref{Bsq 1} at $(0, 0)$

$$ \lambda_{\pm}(\xi)=\pm i |\xi|\left(\frac{(1-\epsilon a|\xi|^2)(1-\epsilon c|\xi|^2)}{(1+\epsilon d|\xi|^2)(1+\epsilon b|\xi|^2)}\right)^{\frac{1}{2}}.$$
are purely imaginary.

\vspace{0.3cm}
This paper will focus on  the exceptional case \eqref{Bsq 3} with $b=d=0,\, a=c=1$ which is the only linearly well-posed case with eigenvalues having non trivial zeroes. Moreover we will restrict to the one-dimensional case, $N=1.$

If $(\zeta,v)$ is a solution of \eqref{Bsq 1}, then by the scaling
\beno
\tilde\z(t,x)=\epsilon\z(\epsilon^{\f12}t,\epsilon^{\f12}x),\quad\tilde{\vv v}(t,x)=\epsilon\vv v(\epsilon^{\f12}t,\epsilon^{\f12}x),
\eeno
 $(\tilde\z,\tilde{\vv v})$ satisfies \eqref{Bsq 1} with $\epsilon=1$ (see also \cite{BCS1}).

\vspace{0.3cm}
In this article, we first establish the long time existence theory for the following strongly dispersive (1D) Boussinesq system
\beq\label{Bsq}\left\{\begin{aligned}
&\p_t\z+(1+\p_x^2)\p_xv+\p_x(\z v)=0,\\
&\p_tv+(1+\p_x^2)\p_x\z+\f12\p_x(v^2)=0,
\end{aligned}\right.\eeq
with initial data
\beq\label{initial}
\z|_{t=0}=\z_0,\quad v|_{t=0}=v_0
\eeq
which are of order $O(\e)$ in a suitable Sobolev class on time scales of order $O(1/{\e^{\f43}})$. {\color{red} } A similar issue was discussed in \cite{IP} for multi-dimensional periodic water waves.

\vspace{0.3cm}
As a consequence, we will prove the  long time existence of solutions to  \eqref{Bsq 1} with $b=d=0, a=c=1$ in the one-dimensional case, that is

\beq\label{Bsqeps}\left\{\begin{aligned}
&\p_t\z+(1+\epsilon\p_x^2)\p_xv+\epsilon\p_x(\z v)=0,\\
&\p_tv+(1+\epsilon\p_x^2)\p_x\z+\f\epsilon 2\p_x(v^2)=0,
\end{aligned}\right.\eeq
with initial data
\beq\label{initialeps}
\z|_{t=0}=\z_0,\quad v|_{t=0}=v_0
\eeq
which are of order $O(1),$ on time scales of order $O(1/\epsilon^{2/3}).$

Contrary to \cite{SX,SWX} where only symmetrization techniques were used to establish the  well-posedness of Boussinesq systems on time scales of order $O(1/\epsilon),$ we will use normal form transformations suitably modified to avoid the zero set of the phases. Normal form techniques  were used  to obtain global or long time existence results of small solutions to the full water wave system, see {\it e.g.,} \cite{AD, IP, Wang}.

\vspace {0.3cm}
We recall that the local well-posedness of \eqref{Bsq} and \eqref{Bsqeps} can be established by reducing to known results for the KdV equation.

Actually, as noticed in \cite{BCS2}, the change of variable $\zeta=u+w, \; v=u-w$ reduces \eqref{Bsqeps} to the following system:
\begin{equation} \label{1Ddiag}
\left\{ \begin{array}{l}    u_t+u_x+\epsilon u_{xxx}
+\frac{\epsilon}{2}\lbrack 3uu_x-ww_x-(uw)_x\rbrack=0 \\
w_t -w_x-\epsilon w_{xxx}+\frac{\epsilon}{2}\lbrack uu_x-3ww_x+(uw)_x\rbrack =0
 \end{array} \right., \quad x \in \mathbb R, \ t \in
\mathbb R,
\end{equation}
which is  a system  of KdV type with uncoupled (diagonal) linear
part.  Thus (see \cite{BCS2}) the Cauchy problem is easily seen to
be locally well-posed for initial data in $H ^s(\R)\times H ^s(\R),$
$s>\frac{3}{4}$ by the results in \cite{KPV1}, \cite{KPV2}.

On the other  hand, as noticed in \cite{ST} Appendix A in a slightly
different context, a minor modification of Bourgain's method as used
in \cite{KPV3} allows to solve the Cauchy problem for \eqref{1Ddiag}
for data in $H^s(\R)\times H^s(\R)$ with $s>-\frac{3}{4} $. We refer
to \cite{BGK} for details. It is worth noticing that in \cite{BGK}
the question of the dependence of the existence time with respect to
$\epsilon$ is not considered but one can check that it is of order $O(1/\sqrt\epsilon).$

By using dispersive properties it has been moreover established in \cite{LPS} that the two-dimensional version of \eqref{Bsqeps} is well-posed in $H^s(\R^2)\times H^s(\R^2)\times H^s(\R^2), s>\frac{3}{2}$ on time scales of order $O(1/\sqrt\epsilon).$ Note that neglecting the dispersive terms in \eqref{Bsqeps} one gets by a standard symmetrization method the existence on time scales of order $O(1/\epsilon)$ but in the "hyperbolic" space $H^s(\R^2), s>2.$

\vspace {0.3cm}
We also recall (see \cite{BCS2}) that \eqref {Bsqeps} and \eqref{Bsq} have an Hamiltonian structure given (for \eqref {Bsqeps} ) by

$$\partial_t \begin{pmatrix}\zeta\\v\end{pmatrix}=J\text{grad}\;H\begin{pmatrix}\zeta\\v\end{pmatrix}$$
where

$$J=\begin{pmatrix}0&\partial_x\\ \partial_x&0\end{pmatrix}$$
and  $$H(\zeta,v)=\frac{1}{2}\int_{-\infty}^\infty (\epsilon\zeta_x^2+\epsilon v_x^2-\zeta^2-v^2-\epsilon v^2\zeta)dx.$$

Unfortunately, contrary to the case $b=d>0, a\leq 0, c<0$ mentioned above, it does not seem possible to use uniquely this structure to prove the global existence of small solutions.

\vspace{0.3cm}
The paper will be organized as follows. The Introduction will continue by some heuristics and the statements of the main results. Section 2 is devoted to some preliminary results. A symmetrization of the strongly dispersive system is given in Section 3 while Sections 4 and 5 are devoted to the proof of the main results,  Theorem \ref{main theorem} and Theorem \ref{long time existence thm} respectively.

\subsection{Heuristics analysis of the system \eqref{Bsq}}\label{Heuristics subsec}
In order to diagonalize  the linear part of \eqref{Bsq}, we define
\beno
V=\z+i\f{\p_x}{|\p_x|}v\quad\text{and}\quad \Lambda=(1+\p_x^2)|\p_x|.
\eeno

Then \eqref{Bsq} is rewritten as

\beq\label{Bsq 4}
\p_tV-i\Lambda V=\sum_{\mu,\nu\in\{+,-\}}Q_{\mu,\nu}(V^\mu,V^\nu),
\eeq
where $V^+=V,\, V^-=\overline{V}$ and $Q_{\mu,\nu}(V^\mu,V^\nu)$ are quadratic terms in $V^\mu$ and $V^\nu$ with symbol $q_{\mu,\nu}(\cdot,\cdot)$, i.e.,
\beq\label{Bsq 5}
\mathcal{F}\Bigl(Q_{\mu,\nu}(V^\mu,V^\nu)\Bigr)(\xi)=\f{1}{2\pi}\int_{\R}q_{\mu,\nu}(\xi,\eta)\widehat{V^\mu}(\xi-\eta)\widehat{V^\nu}(\eta)d\eta.
\eeq
One could check that $|q_{\mu,\nu}(\xi,\eta)|\sim |\xi|$. Since we aim  to prove  long time existence results for solutions  of \eqref{Bsq}, we hope that the quadratic terms could be killed. To do so, we  use  normal form transformation techniques.

Defining the profile of $V$ as follows
\beno
f(t,x)=e^{-it\Lambda} V(t,x),\quad \text{i.e.,}\quad\widehat{f}(t,\xi)=e^{-it\Lambda(\xi)}\widehat{V}(t,\xi),
\eeno
we have
\beq\label{Bsq 6}
\p_t\widehat{f}=\sum_{\mu,\nu\in\{+,-\}}\f{1}{2\pi}\int_{\R}e^{it\Phi_{\mu,\nu}(\xi,\eta)}q_{\mu,\nu}(\xi,\eta)\widehat{f^\mu}(\xi-\eta)\widehat{f^\nu}(\eta)d\eta,
\eeq
where the phase $\Phi_{\mu,\nu}(\xi,\eta)$ is defined by
\beno
\Phi_{\mu,\nu}(\xi,\eta)=-\Lambda(\xi)+\mu\Lambda(\xi-\eta)+\nu\Lambda(\eta).
\eeno
To remove the quadratic terms in the right hand side of \eqref{Bsq 6}, we introduce the following normal forms transformation
\beq\label{Bsq 8}
g=f+\sum_{\mu,\nu\in\{+,-\}}A_{\mu,\nu}(f^\mu,f^\nu),
\eeq
where
\beno
\mathcal{F}\Bigl(A_{\mu,\nu}(f^\mu,f^\nu)\Bigr)(\xi)
=\f{1}{2\pi}\int_{\R}e^{it\Phi_{\mu,\nu}(\xi,\eta)}a_{\mu,\nu}(\xi,\eta)\widehat{f^\mu}(\xi-\eta)\widehat{f^\nu}(\eta)d\eta
\eeno
with the symbol
\beq\label{Bsq 9}
a_{\mu,\nu}(\xi,\eta)=-\f{q_{\mu,\nu}(\xi,\eta)}{i\Phi_{\mu,\nu}(\xi,\eta)}.
\eeq

Thus, we have
\beq\label{Bsq 10}
\p_t\widehat{g}=\f{1}{2\pi}\sum_{\mu,\nu\in\{+,-\}}\int_{\R}e^{it\Phi_{\mu,\nu}(\xi,\eta)}
a_{\mu,\nu}(\xi,\eta)\p_t\Bigl(\widehat{f^\mu}(\xi-\eta)\widehat{f^\nu}(\eta)\Bigr)d\eta.
\eeq

By  virtue of \eqref{Bsq 6}, we see that the r.h.s. of \eqref{Bsq 10} includes the cubic terms in $(f^\mu,f^\nu,f^{\gamma})$. Therefore, if the symbols of quadratic terms have "good" properties, for data of small size  $\e$, the time scale $\f{1}{\e^2}$ is much likely expected.

\medskip

We will use the normal form techniques in another way, that is, integrating by parts with respect to time in the energy estimate. More precisely, energy estimate  gives rise to
\beno
\f12\f{d}{dt}\|V\|_{H^N}^2=\sum_{\mu,\nu\in\{+,-\}}\bigl(Q_{\mu,\nu}(V^\mu,V^\nu)\,|\,V^+\bigr)_{H^N},
\eeno
which implies that
\beq\label{Bsq 19}
\|V(t)\|_{H^N}^2\lesssim\|V(0)\|_{H^N}^2
+\sum_{\mu,\nu\in\{+,-\}}\underbrace{\int_0^t\bigl(Q_{\mu,\nu}(V^\mu,V^\nu)\,|\,V^+\bigr)_{H^N}d\tau}_{I_{\mu,\nu}}.
\eeq

For $I_{\mu,\nu}$, using \eqref{Bsq 5}  and the profiles, we have
\beno\begin{aligned}
&I_{\mu,\nu}=\f{1}{(2\pi)^2}\int_0^t\int_{\R\times\R}\langle\xi\rangle^{2N}q_{\mu,\nu}(\xi,\eta)\widehat{V^\mu}(\xi-\eta)\widehat{V^\nu}(\eta)
\overline{\widehat{V^+}(\xi)}d\eta d\xi d\tau \\
&=\f{1}{(2\pi)^2}\int_0^t\int_{\R\times\R}e^{i\tau\Phi_{\mu,\nu}(\xi,\eta)}\langle\xi\rangle^{2N}q_{\mu,\nu}(\xi,\eta)
\widehat{f^\mu}(\xi-\eta)\widehat{f^\nu}(\eta)
\overline{\widehat{f^+}(\xi)}d\eta d\xi d\tau.
\end{aligned}\eeno
Since
\beno
e^{i\tau\Phi_{\mu,\nu}(\xi,\eta)}=\f{1}{i\Phi_{\mu,\nu}(\xi,\eta)}\f{d}{d\tau}e^{i\tau\Phi_{\mu,\nu}(\xi,\eta)},
\eeno
integrating by parts with respect to $\tau$, we have
\beq\label{Bsq 21}\begin{aligned}
I_{\mu,\nu}=&\f{1}{(2\pi)^2}\int_{\R\times\R}e^{i\tau\Phi_{\mu,\nu}(\xi,\eta)}
\f{\langle\xi\rangle^{2N}q_{\mu,\nu}(\xi,\eta)}{i\Phi_{\mu,\nu}(\xi,\eta)}\cdot
\widehat{f^\mu}(\xi-\eta)\widehat{f^\nu}(\eta)
\overline{\widehat{f^+}(\xi)}d\eta d\xi\bigl|_{\tau=0}^t\\
&-\f{1}{(2\pi)^2}\int_0^t\int_{\R\times\R}e^{i\tau\Phi_{\mu,\nu}(\xi,\eta)}
\f{\langle\xi\rangle^{2N}q_{\mu,\nu}(\xi,\eta)}{i\Phi_{\mu,\nu}(\xi,\eta)}\cdot
\p_\tau\bigl(\widehat{f^\mu}(\xi-\eta)\widehat{f^\nu}(\eta)
\overline{\widehat{f^+}(\xi)}\bigr)d\eta d\xi d\tau.
\end{aligned}\eeq

By virtue of \eqref{Bsq 6}, we see that the second term in the r.h.s of \eqref{Bsq 21} includes inner product between the cubic terms of $(f^\mu,f^\nu,f^{\gamma})$ and $f^+$. The first term in the r.h.s of \eqref{Bsq 21} may be controlled by the initial energy. If the symbols of quadratic terms have  "good" properties, one may derive an energy estimate from \eqref{Bsq 19} so that the time scale $\f{1}{\e^2}$ is much likely expected, provided that  the data is of  small size $\e$.

\medskip

 However, the phase $\Phi_{\mu,\nu}(\xi,\eta)$ may equal $0$ for some $\xi$ and $\eta$. The symbol $a_{\mu,\nu}(\xi-\eta,\eta)$ in \eqref{Bsq 9} is not well-defined for all $(\xi,\eta)\in\R^2$. While the integration by parts with respect to  $\tau$ in \eqref{Bsq 21} could not work for all $(\xi,\eta)\in\R^2$. We have to modify the normal forms transformation only on the "good frequencies set" that is far away from the zeroes of the phase $\Phi_{\mu,\nu}(\xi,\eta)$ . Then the existence time scale may be enlarged. Although we could not obtain the time scale $\f{1}{\e^2}$, we may get the existence time scale $\f{1}{\e^{1+\delta}}$  (for some $\delta\in(0,1)$).  It extends the local existence time scale $\f{1}{\e}$ that can be obtained by  a purely dispersive method as in \cite{LPS}.

\medskip

In the present paper, we thus  use  normal form techniques after integration by parts with respect to time as in \eqref{Bsq 21}.

\subsection{The main results}

We now state the main results of this paper. The first one concerns  the system \eqref{Bsq} without the  small parameter $\epsilon$ but with "small" initial data.
\begin{thm}\label{main theorem}
Assume that  $(\z_0,v_0)\in H^{N_0}(\R)$ for some $N_0\geq 4$ satisfying $\widehat{\z_0}(0)=\widehat{v_0}(0)=0$ and
\beq\label{initial assumption}
\|\z_0\|_{H^{N_0}}^2+\|v_0\|_{H^{N_0}}^2=\e^2.
\eeq
There exists a small $\e_0>0$ such that for all $\e\in(0,\e_0]$, there exists $T_\e=c_0\e^{-\f{4}{3}}$ for some $c_0>0$ and a unique solution $(\z,v)\in C(0,T_\e;H^{N_0}(\R))$ of system \eqref{Bsq}-\eqref{initial} such that
\beq\label{energy estimate}
\sup_{(0,T_\e)}\bigl(\|\z(t)\|_{H^{N_0}}+\|v(t)\|_{H^{N_0}}\bigr)\leq C\bigl(\|\z_0\|_{H^{N_0}}+\|v_0\|_{H^{N_0}}\bigr),
\eeq
where $C>0$ is a universal constant.
\end{thm}
\begin{remark}
If $\widehat{\z_0}(0)=\widehat{v_0}(0)=0$, \eqref{Bsq} shows that $\widehat{\z}(t,0)=\widehat{v}(t,0)=0$ holds for all time $t>0$. Therefore, throughout the whole paper, we shall use the condition   $\widehat{\z}(t,0)=\widehat{v}(t,0)=0$.
\end{remark}

As a consequence of Theorem \ref{main theorem},  we get the long time existence of solutions  to system \eqref{Bsqeps} :

\begin{thm}\label{long time existence thm}
Assume that  $(\z_0,v_0)\in H^{N_0}(\R)$ with $N_0\geq 4$ satisfying $\widehat{\z_0}(0)=\widehat{v_0}(0)=0$.
There exist a small $\epsilon_0>0$ and a constant $T_0=T_0(\|(\z_0,v_0)\|_{H^{N_0}})$ such that for any $\epsilon\in(0, \epsilon_0]$,  there exists a unique solution $(\z,v)\in C(0,T_0\epsilon^{-\f23};H^{N_0}(\R))$ of system \eqref{Bsqeps}-\eqref{initialeps} such that
\beq\label{energy estimate}
\sup_{(0,T_0\epsilon^{-\f23})}\bigl(\|\z(t)\|_{H^{N_0}}+\|v(t)\|_{H^{N_0}}\bigr)\leq C\bigl(\|\z_0\|_{H^{N_0}}+\|v_0\|_{H^{N_0}}\bigr).
\eeq
Here $T_0=T_0(\|(\z_0,v_0)\|_{H^{N_0}})$ is a constant depending on $\|(\z_0,v_0)\|_{H^{N_0}}$.
\end{thm}

\begin{remark}
Contrary to the previous known results on long time existence of other (abcd) Boussinesq systems obtained in \cite{ Bu, Bu2, MSZ, SX, SWX}, we do not reach in Theorem \ref{long time existence thm} the expected time scales $O(1/\epsilon).$ Recall however that Theorem \ref{long time existence thm} improves the $O(1/\sqrt \epsilon)$ result obtained by purely dispersive methods, see \cite{LPS}.
\end{remark}

\subsection{Comments on the proofs of Theorems \ref{main theorem} and \ref{long time existence thm}}
We shall prove two different  long time existence results in Theorems \ref{main theorem} and \ref{long time existence thm}.
The proofs of the theorems share some common features. To avoid losing derivative, we introduce the good unknowns (in the sense of Alinhac \cite{Alin}) $(\z,\, u)$ via nonlinear and nonlocal transformation. Then the principal paralinearization parts for the new system of $V=\z+i\f{\p_x}{|\p_x|}u$ (or $(\z,\,u)$) are symmetric (see \eqref{New Bsq} and \eqref{New Bsq e}).

 However, to enlarge the scale of the existence time, the difficulties of system \eqref{New Bsq} and \eqref{New Bsq e} are different. For system \eqref{New Bsq}, we want to prove an existence time of scale $O(1/{\e^{4/3}})$ when the data are of order $O(\e)$. The main difficulty arises from all the quadratic terms so that we have to deal with  all the quadratic terms by the normal form transformation techniques which sketched in subsection \ref{Heuristics subsec}. Whereas for system \eqref{New Bsq e}, we want to prove an existence time of scale $O(1/{\epsilon^{2/3}})$ when the data are of order $O(1)$ with small parameter $\epsilon$. The key difficulty stems from the quadratic term that is of order $O(\sqrt\epsilon)$ involving the low frequencies.
 We only apply the normal form transformation techniques to such $O(\sqrt\epsilon)$ term. One could check that the normal form transformation could not improve the estimates involving other quadratic terms which are of order $O(\epsilon)$.

\setcounter{equation}{0}
\section{Preliminary}
\subsection{Definitions and notations}

 The notation $f\sim g$ means that there exists a constant $C$ such that $\f{1}{C}f\leq g\leq Cf$.  $f\lesssim g$ means that there exists a constant $C$ such that $f\leq Cg$. We shall use $C$ to denote a universal constant which may changes from line to line. For any $s\in\R$, $H^s(\R)$ denotes the classical  $L^2$ based Sobolev spaces with the norm $\|\cdot\|_{H^s}$.
The notation $\|\cdot\|_{L^p}$ stands for  the $L^p(\R)$ norm for $1\leq p \leq \infty$. For any $k\in\N$, we denote by
\beno
\|f\|_{W^{k,\infty}}=\sum_{j=0}^k\|\p_x^jf\|_{L^\infty}.
\eeno
The $L^2(\R)$ scalar product is denoted by
$(u\,|\,v)_2\eqdefa\int_{\R}u\bar{v}dx$.

If  $A, B$ are two operators, $[A,B]=AB-BA$ denotes their commutator.

The Fourier transform of a tempered distribution $u\in\mathcal{S}'$ is denoted by $\widehat{u}$, which is defined as follows
\beno
\widehat{u}(\xi)\eqdefa\mathcal{F}(u)(\xi)=\int_{\R^n}e^{ix\cdot\xi}u(x)dx.
\eeno
We use $\mathcal{F}^{-1}(f)$ to denote the inverse Fourier  transform of $f(\xi)$.

 If $f$ and $u$ are two functions defined on $\R$, the  Fourier multiplier  $f(D)u$  is defined in term of Fourier transforms, i.e.,
\beno
\widehat{f(D)u}(\xi)=f(\xi)\widehat{u}(\xi).
\eeno

We shall use notations
\beno
\langle\xi\rangle=\bigl(1+|\xi|^2\bigr)^{\f12},\quad\langle\p_x\rangle=\bigl(1+|\p_x|^2\bigr)^{\f12}.
\eeno

For two well-defined functions $f(x)$, $g(x)$ and their bilinear form $Q(f,g)$, we use the convection that the symbol $q(\xi,\eta)$ of $Q(f,g)$  is defined in the following sense
\beno
\mathcal{F}\bigl(Q(f,g)\bigr)(\xi)=\f{1}{2\pi}\int_{\R}q(\xi,\eta)\hat{f}(\xi-\eta)\hat{g}(\xi)d\eta.
\eeno

\subsection{Para-differential decomposition theory}
Our proof of the main results relies on  suitable energy estimates for the solutions of \eqref{Bsq} and \eqref{Bsqeps}. To do so, we introduce  para-differential formulations (see {\it e.g.,} \cite{BCD}) to symmetrize the systems \eqref{Bsq} and \eqref{Bsqeps}.

We fix an even smooth function $\varphi:\,\R\rightarrow [0,1]$ supported in $[-\f32,\f32]$ and equals to 1 in $[-\f54,\f54]$.  For any $k\in\Z$, we define
\beno
\varphi_k(x)\eqdefa\varphi(\f{x}{2^k})-\varphi(\f{x}{2^{k-1}}),\quad \varphi_{\leq k}(x)\eqdefa\varphi(\f{x}{2^k})=\sum_{l\leq k}\varphi_l(x).
\quad  \varphi_{\geq k}(x)\eqdefa1-\varphi_{\leq k-1}(x).
\eeno
While for any interval $I$ of $\R$, we define
\beno
\varphi_I(x)\eqdefa\sum_{k\in I}\varphi_k(x)=\sum_{k\in I\cap\Z}\varphi_k(x).
\eeno
Then for any $x\in\R$,
\beq\label{Bsq 11}
\sum_{k\in\Z}\varphi_k(x)=1\quad\text{and}\quad supp\,\varphi_k(\cdot)\in\{x\in\R\,|\,|x|\in[\f{5}{8}2^k,\f{3}{2}2^k]\}.
\eeq

We use $P_k,\, P_{\leq k}$, $P_{\geq k}$ and $P_I$ to denote the Littlewood-Paley projection operators of the Fourier multiplier $\varphi_k,\,\varphi_{\leq k},\,\varphi_{\geq k}$ and $\varphi_I$, respectively.

We shall use the following para-differential decomposition: for any functions $f,g\in\mathcal{S}'(\R)$,
\beq\label{para-diff decomposition}
fg=T_fg+T_gf+R(f,g),
\eeq
with the para-differential operators being defined as follows
\beno\begin{aligned}
&T_fg=\sum_{j\in\Z}P_{\leq j-7}f\cdot P_jg,\quad R(f,g)=\sum_{j\in\Z}P_jf\cdot P_{[j-6,j+6]}g.
\end{aligned}\eeno

\subsection{Analysis of the phases}
In this subsection, we shall discuss the quadratic phase function
$\Phi_{\mu,\nu}(\xi,\eta)$ which is defined as follows:
\beq\label{quadratic phases}
\Phi_{\mu,\nu}(\xi,\eta)=-\Lambda(\xi)+\mu\Lambda(\xi-\eta)+\nu\Lambda(\eta),\quad\mu,\nu\in\{+,-\},
\eeq
where $\Lambda(\xi)$ is defined by
\beno
\Lambda(\xi)=(1-|\xi|^2)|\xi|=|\xi|-|\xi|^3.
\eeno

We first rewrite the explicit expressions of the phases.
\begin{lemma}\label{phase lem 1}
For any $(\xi,\eta)\in\R^2$ with $\xi\neq\eta,\,\xi\neq0,\,\eta\neq 0$, we have
\beno
\Phi_{+,+}(\xi,\eta)=
\left\{\begin{aligned}
&3|\xi||\xi-\eta||\eta|,\quad\text{if}\quad (\xi-\eta)\cdot\eta>0,\\
&-\f12\min\{|\xi-\eta|,|\eta|\}\bigl(3|\xi|^2+3\max\{|\xi-\eta|^2,|\eta|^2\}+\min\{|\xi-\eta|^2,|\eta|^2\}\bigr)-4\Bigr),\\
&\qquad\qquad\quad\text{if}\quad (\xi-\eta)\cdot\eta<0;
\end{aligned}\right.
\eeno
\beno
\Phi_{-,-}(\xi,\eta)=
\left\{\begin{aligned}
&\f12|\xi|\bigl(|\xi|^2+3|\xi-\eta|^2+3|\eta|^2-4\bigr),\quad\text{if}\quad (\xi-\eta)\cdot\eta>0,\\
&\f12\max\{|\xi-\eta|,|\eta|\}\bigl(3|\xi|^2+3\min\{|\xi-\eta|^2,|\eta|^2\}+\max\{|\xi-\eta|^2,|\eta|^2\}-4\bigr),\\
&\qquad\qquad\qquad\quad\text{if}\quad (\xi-\eta)\cdot\eta<0;
\end{aligned}\right.
\eeno
and
\beno
\Phi_{-,+}(\xi,\eta)=-\Phi_{+,+}(\eta,\xi),\quad \Phi_{+,-}(\xi,\eta)=-\Phi_{+,+}(\eta-\xi,\eta).
\eeno
%
\end{lemma}

\begin{proof}
We derive the expressions of phases one by one.

\medskip

(1) {\it For $\Phi_{+,+}$,} by the definition, we have
\beno\begin{aligned}
\Phi_{+,+}(\xi,\eta)&=-|\xi|+|\xi-\eta|+|\eta|+|\xi|^3-|\xi-\eta|^3-|\eta|^3\\
&=(|\xi|-|\xi-\eta|-|\eta|)\Bigl(2|\xi|^2-\bigl(1+2\text{sign}\bigl((\xi-\eta)\cdot\eta\bigr)\bigr)|\xi-\eta||\eta|\\
&\qquad\qquad+|\xi|(|\xi-\eta|+|\eta|)-1\Bigr)
+3|\xi||\xi-\eta||\eta|.
\end{aligned}\eeno

If $(\xi-\eta)\cdot\eta>0$, we have
\beno
|\xi|-|\xi-\eta|-|\eta|=0,
\eeno
which gives rise to
\beno
\Phi_{+,+}(\xi,\eta)=3|\xi||\xi-\eta||\eta|.
\eeno

If $(\xi-\eta)\cdot\eta<0$, we have
\beno\begin{aligned}
&|\xi|=\bigl||\xi-\eta|-|\eta|\bigr|=\max\{|\xi-\eta|,|\eta|\}-\min\{|\xi-\eta|,|\eta|\},\\
&\text{and}\quad
|\xi|-|\xi-\eta|-|\eta|=-2\min\{|\xi-\eta|,|\eta|\},
\end{aligned}\eeno
which yields
\beno
\begin{aligned}
\Phi_{+,+}(\xi,\eta)
&=-2\min\{|\xi-\eta|,|\eta|\}\Bigl(2|\xi|^2+|\xi-\eta||\eta|+|\xi|(|\xi-\eta|+|\eta|)-\f{3}{2}|\xi|\max\{|\xi-\eta|,|\eta|\}-1\Bigr)\\
&=-2\min\{|\xi-\eta|,|\eta|\}\bigl(\f32\max\{|\xi-\eta|^2,|\eta|^2\}+\min\{|\xi-\eta|^2,|\eta|^2\}-\f32|\xi-\eta||\eta|-1\bigr)\\
&=-2\min\{|\xi-\eta|,|\eta|\}\Bigl(\f34\max\{|\xi-\eta|^2,|\eta|^2\}+\f14\min\{|\xi-\eta|^2,|\eta|^2\}+\f34(|\xi-\eta|-|\eta|)^2-1\Bigr)\\
&=-\f12\min\{|\xi-\eta|,|\eta|\}\bigl(3|\xi|^2+3\max\{|\xi-\eta|^2,|\eta|^2\}+\min\{|\xi-\eta|^2,|\eta|^2\}\bigr)-4\bigr).
\end{aligned}
\eeno

\medskip

(2) {\it For $\Phi_{-,-}$,} by the definition, we have
\beno\begin{aligned}
\Phi_{-,-}(\xi,\eta)&=-\bigl(|\xi|+|\xi-\eta|+|\eta|\bigr)+\bigl(|\xi|^3+|\xi-\eta|^3+|\eta|^3\bigr)\\
&=(|\xi|+|\xi-\eta|+|\eta|)\Bigl(|\xi|^2-|\xi|(|\xi-\eta|+|\eta|)+(|\xi-\eta|+|\eta|)^2-3|\xi-\eta||\eta|-1\Bigr)\\
&\qquad\qquad
+3|\xi||\xi-\eta||\eta|.
\end{aligned}\eeno

If $(\xi-\eta)\cdot\eta>0$, we have
\beno
|\xi|=|\xi-\eta|+|\eta|,
\eeno
and
\beno\begin{aligned}
\Phi_{-,-}(\xi,\eta)&=2|\xi|\bigl(|\xi-\eta|^2+|\eta|^2+\f12|\xi-\eta||\eta|-1\bigr)\\
&=\f12|\xi|\bigl(|\xi|^2+3|\xi-\eta|^2+3|\eta|^2-4\bigr).
\end{aligned}\eeno

If $(\xi-\eta)\cdot\eta<0$, we have
\beno\begin{aligned}
&|\xi|=\bigl||\xi-\eta|-|\eta|\bigr|=\max\{|\xi-\eta|,|\eta|\}-\min\{|\xi-\eta|,|\eta|\},\\
&\text{and}\quad
|\xi|+|\xi-\eta|+|\eta|=2\max\{|\xi-\eta|,|\eta|\},
\end{aligned}\eeno
which implies
\beno\begin{aligned}
\Phi_{-,-}(\xi,\eta)&=2\max\{|\xi-\eta|,|\eta|\}\bigl(\max\{|\xi-\eta|^2,|\eta|^2\}+\f32\min\{|\xi-\eta|^2,|\eta|^2\}-\f32|\xi-\eta||\eta|-1\bigr)\\
&=\f12\max\{|\xi-\eta|,|\eta|\}\bigl(3|\xi|^2+3\min\{|\xi-\eta|^2,|\eta|^2\}+\max\{|\xi-\eta|^2,|\eta|^2\}-4\bigr).
\end{aligned}\eeno

\medskip

(3) {\it For $\Phi_{+,-}$ and $\Phi_{-,+}$,} by the definition, we have
\beno
\Phi_{-,+}(\xi,\eta)=-\Phi_{+,+}(\eta,\xi),\quad\Phi_{+,-}(\xi,\eta)=-\Phi_{+,+}(\eta-\xi,\eta).
\eeno
The lemma is proved.
\end{proof}

%

As a consequence, defining
\beno
\Lambda_\epsilon(\xi)=(1-\epsilon|\xi|^2)|\xi|=|\xi|-\epsilon|\xi|^3,
\eeno
and
\beq\label{quadratic phases for bsq e}
\Phi^\epsilon_{\mu,\nu}(\xi,\eta)=-\Lambda_\epsilon(\xi)+\mu\Lambda_\epsilon(\xi-\eta)+\nu\Lambda_\epsilon(\eta),\quad\mu,\nu\in\{+,-\},
\eeq
we obtain  explicit expressions of the phases $\Phi^\epsilon_{\mu,\nu}(\xi,\eta)$ which  involve the operator $\Lambda_\epsilon$.
\begin{lemma}\label{phase lem for bsq e}
For any $(\xi,\eta)\in\R^2$ with $\xi\neq\eta,\,\xi\neq0,\,\eta\neq 0$, we have
\beno
\Phi_{+,-}(\xi,\eta)=
\left\{\begin{aligned}
&-3\epsilon|\xi||\xi-\eta||\eta|,\quad\text{if}\quad \xi\cdot\eta<0,\\
&\f12\min\{|\xi|,|\eta|\}\bigl(3\epsilon|\xi-\eta|^2+3\epsilon\max\{|\xi|^2,|\eta|^2\}+\epsilon\min\{|\xi|^2,|\eta|^2\}-4\bigr),\\
&\qquad\qquad\qquad\quad\text{if}\quad \xi\cdot\eta>0;
\end{aligned}\right.
\eeno
\beno
\Phi^\epsilon_{-,-}(\xi,\eta)=
\left\{\begin{aligned}
&\f12|\xi|\bigl(\epsilon|\xi|^2+3\epsilon|\xi-\eta|^2+3\epsilon|\eta|^2-4\bigr),\quad\text{if}\quad (\xi-\eta)\cdot\eta>0,\\
&\f12\max\{|\xi-\eta|,|\eta|\}\bigl(3\epsilon|\xi|^2+3\epsilon\min\{|\xi-\eta|^2,|\eta|^2\}+\epsilon\max\{|\xi-\eta|^2,|\eta|^2\}-4\bigr),\\
&\qquad\qquad\qquad\quad\text{if}\quad (\xi-\eta)\cdot\eta<0;
\end{aligned}\right.
\eeno
and
\beno
\Phi^\epsilon_{+,+}(\xi,\eta)=-\Phi^\epsilon_{+,-}(\eta-\xi,\eta),\quad\Phi^\epsilon_{-,+}(\xi,\eta)=-\Phi^\epsilon_{+,+}(\eta,\xi)=\Phi^\epsilon_{+,-}(\xi-\eta,\xi),\quad .
\eeno
\end{lemma}

\subsection{Technical lemmas}

\begin{lemma}\label{tech lem 1}
Let $f,g$ be smooth enough functions. Then,
\beq\label{commutator}
[\f{\p_x}{|\p_x|},\, T_f]g=0.
\eeq
\end{lemma}
\begin{proof}
By the definitions of commutator and  para-differential operators, we have
\beno\begin{aligned}
&\mathcal{F}\Bigl([\f{\p_x}{|\p_x|},\, T_f]g\Bigr)(\xi)=i\int_{\R}\bigl(\text{sign}(\xi)-\text{sign}(\eta)\bigr)
\sum_{j\in\Z}\varphi_{\leq j-7}(|\xi-\eta|)\varphi_j(|\eta|)\hat{f}(\xi-\eta)\hat{g}(\eta)d\eta.
\end{aligned}\eeno
For fixed $\xi\in\R$, when $|\eta|\in(2^k,2^{k+1}]$ with $k\in\Z$, we have
\beq\label{Bsq 18}
\sum_{j\in\Z}\varphi_{\leq j-7}(|\xi-\eta|)\varphi_j(|\eta|)=\sum_{j=k}^{k+1}\varphi_{\leq j-7}(|\xi-\eta|)\varphi_j(|\eta|)\leq\varphi(\f{|\xi-\eta|}{|\eta|}\cdot\f{|\eta|}{2^{k-6}})\leq\varphi_{\leq -6}(\f{|\xi-\eta|}{|\eta|}).
\eeq
Then we get
$
|\xi-\eta|\leq 2^{-5}|\eta|,
$
which yields
\beno
\xi\cdot\eta >0.
\eeno
Otherwise,
\beno
|\xi-\eta|=|\xi|+|\eta|\geq|\eta|.
\eeno
Therefore, we have $\xi\cdot\eta>0$ and
\beno
\mathcal{F}\Bigl([\f{\p_x}{|\p_x|},\, T_f]g\Bigr)(\xi)=0,
\eeno
which implies \eqref{commutator}.
The lemma is proved.
\end{proof}

\setcounter{equation}{0}
\section{Symmetrization of the system \eqref{Bsq}}
In this section, we will symmetrize the system \eqref{Bsq} by introducing good unknowns.

\subsection{Symmetrization of the system \eqref{Bsq}}
By virtue of the para-differential decomposition, we rewrite \eqref{Bsq} to
\beq\label{Bsq 12}\left\{\begin{aligned}
&\p_t\z+(1+\p_x^2)\p_xv+\p_x(T_v\z)+\p_x(T_\z v)+\p_x\bigl(R(\z,v)\bigr)=0,\\
&\p_tv+(1+\p_x^2)\p_x\z+\p_x(T_vv)+\f{1}{2}\p_x\bigl(R(v,v)\bigr)=0.
\end{aligned}\right.\eeq

We introduce good unknowns $(\z,u)$ with
\beq\label{good unknowns}
u=v+B(\z,v),
\eeq
where $B(\cdot,\cdot)$ is a bilinear operator defined as
\beno
B(f,g)=\f12T_f\bigl((1+\p_x^2)^{-1}P_{\geq6}g\bigr).
\eeno
Without confusion, we sometimes use $B$ to denote the bilinear term $B(\z,v)$.

Thanks to \eqref{Bsq 12} and \eqref{good unknowns}, we have
\beno\begin{aligned}
&\p_t\z+(1+\p_x^2)\p_xu+\p_x(T_v\z)+\p_x(T_\z u)=(1+\p_x^2)\p_xB-\p_x\bigl(R(\z,v)\bigr)+\p_x(T_\z B).
\end{aligned}\eeno
Since
\beno
(1+\p_x^2)\p_xB=\f12\p_x\bigl(T_\z P_{\geq6}u\bigr)-\f12\p_x\bigl(T_\z P_{\geq6}B\bigr)+\f12\p_x\Bigl([\p_x^2,T_\z](1+\p_x^2)^{-1}P_{\geq6}v\Bigr),
\eeno
we have
\beq\label{Bsq 13}\begin{aligned}
&\p_t\z+(1+\p_x^2)\p_xu+\p_x(T_v\z)+\f12\p_x(T_\z u)=N_\z,
\end{aligned}\eeq
where
\beno
N_\z=-\f12\p_x\bigl(T_\z P_{\leq5}u\bigr)-\f12\p_x\bigl(T_\z P_{\geq6}B\bigr)+\p_x(T_\z B)+\f12\p_x\Bigl([\p_x^2,T_\z](1+\p_x^2)^{-1}P_{\geq6}v\Bigr)-\p_x\bigl(R(\z,v)\bigr).
\eeno

Using \eqref{Bsq}, \eqref{Bsq 12} and \eqref{good unknowns}, we also have
\beno\begin{aligned}
\p_tu&=\p_tv+B(\p_t\z,v)+B(\z,\p_tv)\\
&=-(1+\p_x^2)\p_x\z-\p_x(T_vv)-\f{1}{2}\p_x\bigl(R(v,v)\bigr)+B(\p_t\z,v)-B(\z,(1+\p_x^2)\p_x\z)-\f12B\bigl(\z,\p_x(|v|^2)\bigr).
\end{aligned}\eeno
Noticing that
\beno\begin{aligned}
&B(\z,(1+\p_x^2)\p_x\z)=\f12T_\z\p_xP_{\geq6}\z,\quad \p_x(T_vv)=\p_x(T_vu)-\p_x(T_vB),
\end{aligned}\eeno
we have
\beq\label{Bsq 14}\begin{aligned}
&\p_tu+(1+\p_x^2)\p_x\z+\p_x(T_vu)+\f12\p_x\bigl(T_\z \z\bigr)=N_u,
\end{aligned}\eeq
where
\beno
N_u=\f12\p_x\bigl(T_\z P_{\leq5}\z\bigr)+\f12T_{\p_x\z}P_{\geq6}\z+\p_x(T_vB)-\f{1}{2}\p_x\bigl(R(v,v)\bigr)+B(\p_t\z,v)-\f12B\bigl(\z,\p_x(|v|^2)\bigr).
\eeno

Now, we define
\beq\label{good variable}
V=\z+i\f{\p_x}{|\p_x|}u.
\eeq
Thanks to \eqref{Bsq 13} and \eqref{Bsq 14}, using \eqref{commutator}, we have
\beq\label{New Bsq}
\p_tV-i\Lambda V+\p_x(T_vV)-\f12i|\p_x|(T_\z V)=N_\z+i\f{\p_x}{|\p_x|}N_u,
\eeq
where $\Lambda=|\p_x|(1-|\p_x|^2)$.
The l.h.s of \eqref{New Bsq} is the quasi-linear part of system \eqref{Bsq}.

 Denoting by
\beno
V^+=V,\quad V^-=\overline{V},
\eeno
we shall rewrite the quadratic terms of \eqref{New Bsq} in terms of $V^+$ and $V^-$. Whereas we keep the cubic and quartic terms in terms of $\z$ and $v$.

Before ending this subsection, we provide a lemma involving the bilinear operator $B(\cdot,\cdot)$.

\begin{lemma}\label{lem for B}
Assume that the real-valued functions $f\in L^\infty(\R)$, $g\in H^s(\R)$ for $s\geq -2$. There hold
\beq\label{estimate for B 1}
\mathcal{F}\bigl(B(f,g)\bigr)(\xi)=\overline{\mathcal{F}\bigl(B(f,g)\bigr)(-\xi)}
\eeq
and
\beq\label{estimate for B}
\|B(f,g)\|_{H^{s+2}}\leq C_B\|f\|_{L^\infty}\|g\|_{H^s},
\eeq
where $C_B>0$ is a universal constant.
\end{lemma}

\begin{proof}
By the definition of $B(\cdot,\cdot)$, we have
\beq\label{Bsq 17}\begin{aligned}
&\mathcal{F}\bigl(B(f,g)\bigr)(\xi)=\f{1}{4\pi}\int_{\R}\widehat{f}(\xi-\eta)\widehat{g}(\eta)(1-|\eta|^2)^{-1}
\varphi_{\geq6}(|\eta|)\sum_{j\in\Z}\varphi_{\leq j-7}(|\xi-\eta|)\varphi_j(|\eta|)d\eta,
\end{aligned}\eeq
and
\beno\begin{aligned}
&\overline{\mathcal{F}\bigl(B(f,g)\bigr)(-\xi)}=\f{1}{4\pi}\int_{\R}\overline{\widehat{f}(-\xi-\eta)}\,
\overline{\widehat{g}(\eta)}(1-|\eta|^2)^{-1}
\varphi_{\geq6}(|\eta|)\sum_{j\in\Z}\varphi_{\leq j-7}(|-\xi-\eta|)\varphi_j(|\eta|)d\eta\\
&=\f{1}{4\pi}\int_{\R}\overline{\widehat{f}(-\xi+\eta)}\,\overline{\widehat{g}(-\eta)}(1-|\eta|^2)^{-1}
\varphi_{\geq6}(|\eta|)\sum_{j\in\Z}\varphi_{\leq j-7}(|\xi-\eta|)\varphi_j(|\eta|)d\eta.
\end{aligned}\eeno
Since  $f,g$ are real-valued functions, we have
\beno
\overline{\widehat{f}(-\xi+\eta)}=\widehat{f}(\xi-\eta),\quad\overline{\widehat{g}(-\eta)}=\widehat{g}(\eta),
\eeno
which gives rise to
\beno
\overline{\mathcal{F}\bigl(B(f,g)\bigr)(-\xi)}=\f{1}{4\pi}\int_{\R}\widehat{f}(\xi-\eta)\widehat{g}(\eta)(1-|\eta|^2)^{-1}
\varphi_{\geq6}(|\eta|)\sum_{j\in\Z}\varphi_{\leq j-7}(|\xi-\eta|)\varphi_j(|\eta|)d\eta.
\eeno
Then we have
\beno
\mathcal{F}\bigl(B(f,g)\bigr)(\xi)=\overline{\mathcal{F}\bigl(B(f,g)\bigr)(-\xi)}.
\eeno

Estimate \eqref{estimate for B} follows from the standard estimate on  $T_fg$ and the definition of $B(f,g)$.
This completes the proof of the lemma.
\end{proof}

\subsection{Main proposition for the symmetric system \eqref{New Bsq}}
For \eqref{New Bsq}, we state the following proposition.
\begin{proposition}\label{New Bsq prop}
Assume that $(\z,v)\in H^{N_0}(\R)$ with $N_0\geq 4$ solves \eqref{Bsq}. Then $V$ defined in \eqref{good variable} satisfies the following system
\beq\label{New formula}
\p_tV-i\Lambda V=\mathcal{S}_V+\mathcal{Q}_V+\mathcal{R}_V+\mathcal{M}_V+\mathcal{L}_V+\mathcal{C}_V+\mathcal{N}_V,
\eeq
where
\begin{itemize}
\item The quadratic term $\mathcal{S}_V$ is of the form
\beno
\mathcal{S}_V=S_{+,+}(V^+,V^+)+S_{-,+}(V^-,V^+).
\eeno
And the symbol $s_{\mu,+}(\xi,\eta)$ of $S_{\mu,+}$ (for $\mu=+,-$) satisfies
\beq\label{symbol s}
\overline{s_{\mu,+}(\xi,\eta)}=-s_{\mu,+}(\xi,\eta),
\eeq
\beq\label{bound of s}\begin{aligned}
&|\langle\xi\rangle^{-N_0}\langle\eta\rangle^{-N_0}\bigl(\langle\xi\rangle^{2N_0}s_{\mu,+}(\xi,\eta)
-\langle\eta\rangle^{2N_0}s_{-\mu,+}(\eta,\xi)\bigr)|\lesssim|\xi-\eta|\cdot\varphi_{\leq -6}\Bigl(\f{|\xi-\eta|}{\max\{|\xi|,|\eta|\}}\Bigr).
\end{aligned}\eeq

\item The quadratic term $\mathcal{Q}_V$ is of the form
\beno
\mathcal{Q}_V=Q_{+,-}(V^+,V^-)+Q_{-,-}(V^-,V^-).
\eeno
And the symbol $q_{\mu,-}(\xi,\eta)$ of $Q_{\mu,-}$ satisfies
\beq\label{bound of q}\begin{aligned}
&|q_{\mu,-}(\xi,\eta)|\lesssim|\xi|\cdot\varphi_{\leq 5}\bigl(|\eta|\bigr)\cdot\varphi_{\leq -6}\Bigl(\f{|\xi-\eta|}{|\eta|}\Bigr).
\end{aligned}\eeq

\item The quadratic term $\mathcal{R}_V$ is of the form
\beno
\mathcal{R}_V=\sum_{\mu,\nu\in\{+,-\}}R_{\mu,\nu}(V^\mu,V^\nu).
\eeno
And the symbol $r_{\mu,\nu}(\xi,\eta)$ of $R_{\mu,\nu}$ satisfies
\beq\label{bound of r}\begin{aligned}
&|r_{\mu,\nu}(\xi,\eta)|\lesssim|\xi-\eta|\cdot\varphi_{\geq 6}\bigl(|\eta|\bigr)\cdot \varphi_{\leq -6}\Bigl(\f{|\xi-\eta|}{|\eta|}\Bigr).
\end{aligned}\eeq

\item The quadratic term $\mathcal{M}_V$ is of the form
\beno
\mathcal{M}_V=\sum_{\mu,\nu\in\{+,-\}}M_{\mu,\nu}(V^\mu,V^\nu).
\eeno
And  the symbol $m_{\mu,\nu}(\xi,\eta)$ of $M_{\mu,\nu}$ satisfies
\beq\label{bound of m}
|m_{\mu,\nu}(\xi,\eta)|\lesssim|\xi|\cdot\varphi_{[-6,7]}\Bigl(\f{|\xi-\eta|}{|\eta|}\Bigr).
\eeq

\item The cubic term $\mathcal{L}_V=\p_x(T_BV)$
satisfies
\beq\label{estimate for cubic term 1}
\bigl|\text{Re}\bigl\{\bigl(\langle\p_x\rangle^{N_0}\mathcal{L}_V\,|\,\langle\p_x\rangle^{N_0}V\bigr)_2\bigr\}\bigr|
\lesssim\|\z\|_{L^\infty}\|v\|_{L^2}\|V\|_{H^{N_0}}^2.
\eeq

\item The cubic term $\mathcal{C}_V$ satisfies
\beq\label{estimate for cubic term 2}
\|\mathcal{C}_V\|_{H^{N_0}}\lesssim\|\z\|_{W^{1,\infty}}\bigl(\|\z\|_{H^{N_0}}^2+\|v\|_{H^{N_0}}^2\bigr).
\eeq

\item The quartic term $\mathcal{N}_V$ satisfies
\beq\label{estimate for quartic term}
\|\mathcal{N}_V\|_{H^{N_0}}\lesssim\|\z\|_{L^\infty}^2\|v\|_{H^{N_0}}^2.
\eeq
\end{itemize}
\end{proposition}

\begin{remark}
Proposition \ref{New Bsq prop} shows that there is no loss of derivative  for the nonlinear terms of \eqref{New Bsq}. Indeed, in the energy estimates, we shall use the symmetric structure of the quadratic terms $\mathcal{S}_V$ to avoid losing derivative(see \eqref{bound of s}). We also use the symmetric structure of $\mathcal{L}_V$ to avoid losing derivative(see the proof of \eqref{estimate for cubic term 1}).
\end{remark}

\begin{remark}
 For the symmetric system \eqref{New Bsq} or \eqref{Bsq 13}-\eqref{Bsq 14}, the standard energy estimates will provide the local existence on  time of scale $O(\f{1}{\e}),$  for the initial data of size $\e$.  To enlarge the existence time of the system, we shall use the new formulation \eqref{New formula} and the normal form transformations. Thanks to Proposition \ref{New Bsq prop}, if the quadratic terms equal zero,  the estimates of  the cubic terms and the quartic terms guarantee the existence time of scale $\f{1}{\e^2}$. For the non-trivial quadratic terms in \eqref{New formula},  we shall apply normal forms transformation in the "good frequencies set"(far away from zeroes of the phases ) to kill the quadratic terms to the cubic and quartic order terms, while for the quadratic terms  in the  "bad frequencies set" (in a small neighborhood of zeroes of phases), we will use the smallness size of the frequencies set. Combining the two estimates on both "good" and "bad" sets, the optimal "cut-off" of the frequencies spaces will determine an existence time of order $\f{1}{\e^{4/3}}$ for $\e$ sufficiently small.
 \end{remark}

\subsection{Proof of Proposition \ref{New Bsq prop}} In this subsection, we present the proof of  Proposition \ref{New Bsq prop}.
\begin{proof}[ Proof of Proposition \ref{New Bsq prop}]
The nonlinear terms in the r.h.s. of \eqref{New formula} come from the nonlinear terms in \eqref{New Bsq}. We rewrite \eqref{New Bsq} to be \eqref{New formula} with the nonlinear terms in the following forms
\beno\begin{aligned}
&\mathcal{S}_V=-\p_x(T_uV)+\f12i|\p_x|(T_\z V),\\
&\mathcal{Q}_V=-\f12\p_x\bigl(T_\z P_{\leq5}u\bigr)-\f{i}{2}|\p_x|\bigl(T_\z P_{\leq5}\z\bigr),\\
&\mathcal{R}_V=\f12\p_x\Bigl([\p_x^2,T_\z](1+\p_x^2)^{-1}P_{\geq6}u\Bigr)+\f{i}{2}\f{\p_x}{|\p_x|}\bigl(T_{\p_x\z}P_{\geq6}\z\bigr)
-\f{i\p_x}{|\p_x|}B((1+\p_x^2)\p_x u,u),\\
&\mathcal{M}_V=-\p_x\bigl(R(\z,u)\bigr)+\f{i}{2}|\p_x|\bigl(R(u,u)\bigr),\\
&\mathcal{L}_V=\p_x(T_BV),\\
&\mathcal{C}_V=\p_x(T_\z B)-\f12\p_x\bigl(T_\z P_{\geq6}B\bigr)-\f12\p_x\bigl([\p_x^2,T_\z](1+\p_x^2)^{-1}P_{\geq6}B\bigr)-i|\p_x|(T_vB)\\
&\qquad+\f{i\p_x}{|\p_x|}B\bigl((1+\p_x^2)\p_x v,B\bigr)+\f{i\p_x}{|\p_x|}B\bigl((1+\p_x^2)\p_x B,v\bigr)-\f{i\p_x}{|\p_x|}B\bigl(\p_x(\z v),v\bigr)
+\p_x\bigl(R(\z,B)\bigr)\\
&\qquad-\f{i}{2}|\p_x|\bigl(R(v,B)\bigr)-\f{i}{2}|\p_x|\bigl(R(B,v)\bigr)
-\f{i}{2}\f{\p_x}{|\p_x|}B\bigl(\z,\p_x(|v|^2)\bigr),\\
&\mathcal{N}_V=\f{i\p_x}{|\p_x|}B\bigl((1+\p_x^2)\p_x B,B\bigr)-\f{i}{2}|\p_x|\bigl(R(B,B)\bigr).
\end{aligned}\eeno
Here we used the first equation of \eqref{Bsq} and \eqref{good unknowns}. Thanks to \eqref{good variable}, we have
\beq\label{Bsq 15}
\z=\f{1}{2}(V^++V^-)=\f12\sum_{\mu\in\{+,-\}}V^\mu,\quad u=\f{i}{2}\f{\p_x}{|\p_x|}(V^+-V^-)=\f{i}{2}\sum_{\mu\in\{+,-\}}\mu\f{\p_x}{|\p_x|}V^\mu.
\eeq

\medskip

(1) {\it For the quadratic term $\mathcal{S}_V$,}
by virtue of \eqref{Bsq 15}, we rewrite it in terms of $V^+$ and $V^-$ as
\beno
\mathcal{S}_V=S_{+,+}(V^+,V^+)+S_{-,+}(V^-,V^+),
\eeno
with
\beno\begin{aligned}
&S_{\mu,+}(V^\mu,V^+)=-\mu\f{i}{2}\p_x(T_{\f{\p_x}{|\p_x|}V^\mu}V^+)+\f14i|\p_x|(T_{V^\mu} V^+).
\end{aligned}\eeno
By the definition of the para-differential operator, we have
\beno
\mathcal{F}\Bigl(S_{\mu,+}(V^\mu,V^+)\Bigr)(\xi)=\f{1}{2\pi}\int_{\R}i\Bigl(\mu\f12\xi\,\text{sign}(\xi-\eta)+\f{1}{4}|\xi|\Bigr)
\sum_{j\in\Z}\varphi_{\leq j-7}(|\xi-\eta|)\varphi_j(|\eta|)\widehat{V^\mu}(\xi-\eta)\widehat{V^+}(\eta)d\eta,
\eeno
with the symbol
\beq\label{symbol s a}
s_{\mu,+}(\xi,\eta)=i\Bigl(\mu\f12\xi\,\text{sign}(\xi-\eta)+\f{1}{4}|\xi|\Bigr)
\sum_{j\in\Z}\varphi_{\leq j-7}(|\xi-\eta|)\varphi_j(|\eta|).
\eeq
Then \eqref{symbol s a} yields $\overline{s_{\mu,+}(\xi,\eta)}=-s_{\mu,+}(\xi,\eta)$ which is exactly \eqref{symbol s}.
Thanks to \eqref{Bsq 18}, we have
\beno
\sum_{j\in\Z}\varphi_{\leq j-7}(|\xi-\eta|)\varphi_j(|\eta|)\lesssim\varphi_{\leq -6}\Bigl(\f{|\xi-\eta|}{|\eta|}\Bigr),
\eeno
which implies
\beq\label{relation between xi and eta}
\xi\cdot\eta>0\quad\text{and}\quad \f{31}{32}|\eta|\leq|\xi|\leq\f{33}{32}|\eta|.
\eeq
Since
\beno\begin{aligned}
&\quad\langle\xi\rangle^{-N_0}\langle\eta\rangle^{-N_0}\bigl(\langle\xi\rangle^{2N_0}s_{\mu,+}(\xi,\eta)
-\langle\eta\rangle^{2N_0}s_{-\mu,+}(\eta,\xi)\bigr)\\
&=i\langle\xi\rangle^{-N_0}\langle\eta\rangle^{-N_0}\Bigl\{\langle\xi\rangle^{2N_0}
\Bigl(\mu\f12\xi\,\text{sign}(\xi-\eta)+\f{1}{4}|\xi|\Bigr)\sum_{j\in\Z}\varphi_{\leq j-7}(|\xi-\eta|)\varphi_j(|\eta|)\\
&\qquad
-\langle\eta\rangle^{2N_0}\Bigl(\mu\f12\eta\,\text{sign}(\xi-\eta)+\f{1}{4}|\eta|\Bigr)\sum_{j\in\Z}\varphi_{\leq j-7}(|\xi-\eta|)\varphi_j(|\xi|)\Bigr\},
\end{aligned}\eeno
and
\beno\begin{aligned}
&|\varphi_j(|\xi|)-\varphi_j(|\eta|)|\lesssim\f{1}{\min\{|\xi|,|\eta|\}}|\xi-\eta|,
\end{aligned}\eeno
using \eqref{relation between xi and eta}, we obtain
\beno
|\langle\xi\rangle^{-N_0}\langle\eta\rangle^{-N_0}\bigl(\langle\xi\rangle^{2N_0}s_{\mu,+}(\xi,\eta)
-\langle\eta\rangle^{2N_0}s_{-\mu,+}(\eta,\xi)\bigr)|\lesssim|\xi-\eta|\cdot\varphi_{\leq -6}\Bigl(\f{|\xi-\eta|}{\max\{|\xi|,|\eta|\}}\Bigr).
\eeno
This is exactly \eqref{bound of s}.
\medskip

(2) {\it For the quadratic term $\mathcal{Q}_V$,}
by virtue of \eqref{Bsq 15}, we rewrite it in terms of $V^+$ and $V^-$ as
\beno
\mathcal{Q}_V=\sum_{\mu,\nu\in\{+,-\}}Q_{\mu,\nu}(V^\mu,V^\nu),
\eeno
where
\beno
Q_{\mu,\nu}(V^\mu,V^\nu)=-\nu\f{i}{8}\p_x\bigl(T_{V^\mu}\f{\p_x}{|\p_x|}P_{\leq5}V^\nu\bigr)
-\f{i}{8}|\p_x|\bigl(T_{V^\mu}P_{\leq5}V^\nu\bigr).
\eeno
Applying Fourier transformation to $Q_{\mu,\nu}(V^\mu,V^\nu)$, we have
\beno\begin{aligned}
\mathcal{F}\Bigl(Q_{\mu,\nu}(V^\mu,V^\nu)\Bigr)(\xi)=\f{1}{2\pi}\int_{\R}\f{i}{8}\bigl(\nu\xi\text{sign}(\eta)-|\xi|\bigr)\varphi_{\leq5}(|\eta|)
\sum_{j\in\Z}\varphi_{\leq j-7}(|\xi-\eta|)\varphi_j(|\eta|)\widehat{V^\mu}(\xi-\eta)\widehat{V^\nu}(\eta)d\eta.
\end{aligned}\eeno
Using \eqref{relation between xi and eta}, we have the symbol of $Q_{\mu,\nu}(V^\mu,V^\nu)$  as follows
\beq\label{symbol q}
q_{\mu,\nu}(\xi,\eta)=\f{i}{8}\bigl(\nu|\xi|-|\xi|\bigr)\varphi_{\leq5}(|\eta|)
\sum_{j\in\Z}\varphi_{\leq j-7}(|\xi-\eta|)\varphi_j(|\eta|)
\eeq
Thanks to \eqref{Bsq 18}, we obtain
\beno
q_{\mu,+}(\xi,\eta)=0\quad\text{and}\quad|q_{\mu,-}(\xi,\eta)|\lesssim |\xi|\varphi_{\leq5}(|\eta|)\varphi_{\leq-6}\Bigl(\f{|\xi-\eta|}{|\eta|}\Bigr),
\eeno
which implies $Q_{\mu,+}=0$ and \eqref{bound of q}.

\medskip

(3) {\it For the quadratic term $\mathcal{R}_V$,}
by  virtue of \eqref{Bsq 15}, we rewrite it in terms of $V^+$ and $V^-$ as
\beno
\mathcal{R}_V=\sum_{\mu,\nu\in\{+,-\}}R_{\mu,\nu}(V^\mu,V^\nu),
\eeno
where
\beno\begin{aligned}
&R_{\mu,\nu}(V^\mu,V^\nu)=
\nu\f{i}{8}\p_x\Bigl([\p_x^2,T_{V^\mu}](1+\p_x^2)^{-1}\f{\p_x}{|\p_x|}P_{\geq6}V^\nu\Bigr)\\
&\qquad
+\f{i}{8}\f{\p_x}{|\p_x|}\bigl(T_{\p_xV^\mu}P_{\geq6}V^\nu\bigr)
-\mu\nu\f{i}{8}\f{\p_x}{|\p_x|}\bigl(T_{(1+\p_x^2)|\p_x|V^\mu}(1+\p_x^2)^{-1}\f{\p_x}{|\p_x|}P_{\geq6}V^\nu\bigr).
\end{aligned}\eeno
Here we used the definition of $B(\cdot,\cdot)$. Applying Fourier transformation to  $R_{\mu,\nu}(V^\mu,V^\nu)$, we have
\beno\begin{aligned}
&\mathcal{F}\Bigl(R_{\mu,\nu}(V^\mu,V^\nu)\Bigr)(\xi)=\f{1}{2\pi}\int_{\R}\f{i}{8}\Bigl(\nu\bigl(|\xi|^2-|\eta|^2\bigr)(1-|\eta|^2)^{-1}\xi\text{sign}(\eta)\\
&\qquad-(\xi-\eta)\text{sign}(\xi)
+\mu\nu(1-|\xi-\eta|^2)|\xi-\eta|(1-|\eta|^2)^{-1}\text{sign}(\xi)\text{sign}(\eta)\Bigr)\\
&\qquad\cdot\varphi_{\geq6}(|\eta|)\cdot
\sum_{j\in\Z}\varphi_{\leq j-7}(|\xi-\eta|)\varphi_j(|\eta|)\widehat{V^\mu}(\xi-\eta)\widehat{V^\nu}(\eta)d\eta,
\end{aligned}\eeno
which implies that the symbol of $R_{\mu,\nu}(\cdot,\cdot)$ is
\beno\begin{aligned}
&r_{\mu,\nu}(\xi,\eta)=\f{i}{8}\Bigl(\nu\bigl(|\xi|^2-|\eta|^2\bigr)(1-|\eta|^2)^{-1}|\xi|
-(\xi-\eta)\text{sign}(\xi)\\
&\qquad
+\mu\nu(1-|\xi-\eta|^2)|\xi-\eta|(1-|\eta|^2)^{-1}\Bigr)\varphi_{\geq6}(|\eta|)\sum_{j\in\Z}\varphi_{\leq j-7}(|\xi-\eta|)\varphi_j(|\eta|),
\end{aligned}\eeno
where we used the fact $\xi\cdot\eta>0$ (in \eqref{relation between xi and eta}). Using \eqref{Bsq 18}, we obtain \eqref{bound of r}.

\medskip

(3) {\it For the quadratic term $\mathcal{M}_V$,}
by virtue of \eqref{Bsq 15}, we rewrite it in terms of $V^+$ and $V^-$ as
\beno
\mathcal{M}_V=\sum_{\mu,\nu\in\{+,-\}}M_{\mu,\nu}(V^\mu,V^\nu),
\eeno
where
\beno\begin{aligned}
&M_{\mu,\nu}(V^\mu,V^\nu)=-\nu\f{i}{4}\p_x\bigl(R(V^\mu,\f{\p_x}{|\p_x|}V^\nu)\bigr)-\mu\nu\f{i}{8}|\p_x|\bigl(R(\f{\p_x}{|\p_x|}V^\mu,\f{\p_x}{|\p_x|}V^\nu)\bigr).
\end{aligned}\eeno
Then we have
\beno\begin{aligned}
&\mathcal{F}\Bigl(M_{\mu,\nu}(V^\mu,V^\nu)\Bigr)(\xi)=\f{1}{2\pi}\int_{\R}\f{i}{8}\Bigl(2\nu\xi\,\text{sign}(\eta)+\mu\nu|\xi|\text{sign}(\xi-\eta)\text{sign}(\eta)\Bigr)\\
&\qquad\times
\sum_{j\in\Z}\varphi_j(|\xi-\eta|)\varphi_{[j-6,j+6]}(|\eta|)\widehat{V^\mu}(\xi-\eta)\widehat{V^\nu}(\eta)d\eta.
\end{aligned}\eeno
and we obtain the symbol of $M_{\mu,\nu}(\cdot,\cdot)$ as follows
\beno\begin{aligned}
&m_{\mu,\nu}(\xi,\eta)=\f{i}{8}\Bigl(2\nu\xi\text{sign}(\eta)+\mu\nu|\xi|\text{sign}(\xi-\eta)\text{sign}(\eta)\Bigr)
\sum_{j\in\Z}\varphi_j(|\xi-\eta|)\varphi_{[j-6,j+6]}(|\eta|).
\end{aligned}\eeno

For fixed $\xi$, when $|\eta|\in(2^k,2^{k+1}]$ with $k\in\Z$, we have
\beno\begin{aligned}
&\sum_{j\in\Z}\varphi_j(|\xi-\eta|)\varphi_{[j-6,j+6]}(|\eta|)=\sum_{j=k-6}^{k+7}\varphi_j(|\xi-\eta|)\varphi_{[j-6,j+6]\cap[k,k+1]}(|\eta|)\\
&\leq\varphi_{[k-6,k+7]}(|\xi-\eta|)\leq \varphi_{[-6,7]}\Bigl(\f{|\xi-\eta|}{|\eta|}\Bigr),
\end{aligned}
\eeno
which implies \eqref{bound of m}.

\medskip

(4) {\it For the cubic term $\mathcal{L}_V$,}  we first have
\beno\begin{aligned}
&\widehat{\mathcal{L}_V}(\xi)=\f{i}{2\pi}\int_{\R}\xi\sum_{j\in\Z}\varphi_{\leq j-7}(|\xi-\eta|)\varphi_j(|\eta|)\widehat{B}(\xi-\eta)\widehat{V}(\eta)d\eta.
\end{aligned}\eeno
Then there holds
\beno\begin{aligned}
&\bigl(\langle\p_x\rangle^{N_0}\mathcal{L}_V\,|\,\langle\p_x\rangle^{N_0}V\bigr)_2
=\f{1}{2\pi}\int_{\R}\langle\xi\rangle^{2N_0}\widehat{\mathcal{L}_V}(\xi)\overline{\widehat{V}(\xi)}d\xi\\
&=\f{i}{(2\pi)^2}\int_{\R^2}\xi\langle\xi\rangle^{2N_0}\sum_{j\in\Z}\varphi_{\leq j-7}(|\xi-\eta|)\varphi_j(|\eta|)\widehat{B}(\xi-\eta)\widehat{V}(\eta)\overline{\widehat{V}(\xi)}d\eta d\xi
\end{aligned}\eeno
and
\beno\begin{aligned}
&\overline{\bigl(\langle\p_x\rangle^{N_0}\mathcal{L}_V\,|\,\langle\p_x\rangle^{N_0}V\bigr)_2}\\
&=-\f{i}{(2\pi)^2}\int_{\R^2}\xi\langle\xi\rangle^{2N_0}\sum_{j\in\Z}\varphi_{\leq j-7}(|\xi-\eta|)\varphi_j(|\eta|)\overline{\widehat{B}(\xi-\eta)}\widehat{V}(\xi)\overline{\widehat{V}(\eta)}d\eta d\xi\\
&=-\f{i}{(2\pi)^2}\int_{\R^2}\eta\langle\eta\rangle^{2N_0}\sum_{j\in\Z}\varphi_{\leq j-7}(|\xi-\eta|)\varphi_j(|\xi|)\overline{\widehat{B}(\eta-\xi)}\widehat{V}(\eta)\overline{\widehat{V}(\xi)}d\eta d\xi.
\end{aligned}\eeno
Thanks to \eqref{estimate for B 1}, we have
\beno
\overline{\widehat{B}(\eta-\xi)}=\widehat{B}(\xi-\eta)
\eeno
which leads to
\beno\begin{aligned}
&\text{Re}\bigl\{\bigl(\langle\p_x\rangle^{N_0}\mathcal{L}_V\,|\,\langle\p_x\rangle^{N_0}V\bigr)_2\bigr\}
=\f{1}{2}\Bigl(\bigl(\langle\p_x\rangle^{N_0}\mathcal{L}_V\,|\,\langle\p_x\rangle^{N_0}V\bigr)_2
+\overline{\bigl(\langle\p_x\rangle^{N_0}\mathcal{L}_V\,|\,\langle\p_x\rangle^{N_0}V\bigr)_2}\Bigr)\\
&=\f{i}{8\pi^2}\int_{\R^2}l(\xi,\eta)\widehat{B}(\xi-\eta)\cdot\langle\eta\rangle^{N_0}\widehat{V}(\eta)
\cdot\langle\xi\rangle^{N_0}\overline{\widehat{V}(\xi)}d\eta d\xi,
\end{aligned}\eeno
where
\beno
l(\xi,\eta)=\sum_{j\in\Z}\varphi_{\leq j-7}(|\xi-\eta|)\bigl(\xi\langle\xi\rangle^{2N_0}\varphi_j(|\eta|)
-\eta\langle\eta\rangle^{2N_0}\varphi_j(|\xi|)\bigr)\langle\xi\rangle^{-N_0}\langle\eta\rangle^{-N_0}.
\eeno

Using \eqref{Bsq 18} for $l(\xi,\eta)$, there holds
\beno
|\xi|\sim|\eta|, \quad \xi\cdot\eta>0.
\eeno
For fixed $\xi,\eta$, we have
\beno
|\varphi_j(|\xi|)-\varphi_j(|\eta|)|\lesssim\f{1}{\min\{|\xi,|\eta||\}}|\xi-\eta|\lesssim\f{1}{|\xi|}|\xi-\eta|.
\eeno
Noticing that for fixed $\xi,\eta$, the summation in $l(\xi,\eta)$ is finite, we get
\beno
|l(\xi,\eta)|\lesssim|\xi-\eta|.
\eeno
Then we obtain
\beno\begin{aligned}
&\bigl|\text{Re}\bigl\{\bigl(\langle\p_x\rangle^{N_0}\mathcal{L}_V\,|\,\langle\p_x\rangle^{N_0}V\bigr)_2\bigr\}\bigr|
\lesssim\int_{\R^2}|\xi-\eta||\widehat{B}(\xi-\eta)|\cdot\langle\eta\rangle^{N_0}|\widehat{V}(\eta)|
\cdot\langle\xi\rangle^{N_0}|\widehat{V}(\xi)|d\eta d\xi\\
&\lesssim\|\xi\widehat{B}(\xi)\|_{L^1}\|\langle\xi\rangle^{N_0}\widehat{V}(\xi)\|_{L^2}^2
\lesssim\|\langle\xi\rangle^2\widehat{B}(\xi)\|_{L^2}\|V\|_{H^{N_0}}^2\lesssim\|B\|_{H^2}\|V\|_{H^{N_0}}^2.
\end{aligned}\eeno

Thanks to \eqref{estimate for B}, we have
\beno
\|B\|_{H^2}=\|B(\z,v)\|_{H^2}\lesssim\|\z\|_{L^\infty}\|v\|_{L^2}.
\eeno
Thus we obtain
\beno
\bigl|\text{Re}\bigl\{\bigl(\langle\p_x\rangle^{N_0}\mathcal{L}_V\,|\,\langle\p_x\rangle^{N_0}V\bigr)_2\bigr\}\bigr|
\lesssim\|\z\|_{L^\infty}\|v\|_{L^2}\|V\|_{H^{N_0}}^2.
\eeno
This is \eqref{estimate for cubic term 1}.

\medskip

(5) {\it For the cubic term $\mathcal{C}_V$,} we first have
\beq\label{Bsq 20}\begin{aligned}
&\|\mathcal{C}_V\|_{H^{N_0}}\lesssim\bigl(\|\z\|_{W^{1,\infty}}+\|v\|_{W^{1,\infty}}\bigr)\|B\|_{H^{N_0+1}}
+\|\p_x\bigl([\p_x^2,T_\z](1+\p_x^2)^{-1}P_{\geq6}B\bigr)\|_{H^{N_0}}\\
&\quad
+\|B\bigl((1+\p_x^2)\p_x v,B\bigr)\|_{H^{N_0}}
+\|B\bigl((1+\p_x^2)\p_x B,v\bigr)\|_{H^{N_0}}\\
&\quad
+\|B\bigl(\p_x(\z v),v\bigr)\|_{H^{N_0}}
+\|B\bigl(\z,\p_x(|v|^2)\bigr)\|_{H^{N_0}}.
\end{aligned}\eeq
Since
\beno
[\p_x^2, T_f]g=2T_{\p_xf}\p_xg+T_{\p_x^2f}g,
\eeno
we have
\beno\begin{aligned}
&\|\p_x\bigl([\p_x^2,T_\z](1+\p_x^2)^{-1}P_{\geq6}B\bigr)\|_{H^{N_0}}
\lesssim\|T_{\p_x\z}(1+\p_x^2)^{-1}\p_x^2P_{\geq6}B\|_{H^{N_0}}\\
&\quad+\|T_{\p_x^2\z}(1+\p_x^2)^{-1}\p_xP_{\geq6}B\|_{H^{N_0}}
+\|T_{\p_x^3\z}(1+\p_x^2)^{-1}P_{\geq6}B\|_{H^{N_0}}\\
&\lesssim\|\z\|_{W^{3,\infty}}\|B\|_{H^{N_0}}.
\end{aligned}\eeno
Thanks to \eqref{estimate for B}, we can bound the last four terms  of \eqref{Bsq 20} by
\beno\begin{aligned}
&\lesssim\|(1+\p_x^2)\p_x v\|_{L^\infty}\|B\|_{H^{N_0-2}}+\|(1+\p_x^2)\p_x B\|_{L^\infty}\|v\|_{H^{N_0-2}}
+\|\p_x(\z v)\|_{L^\infty}\|v\|_{H^{N_0-2}}\\
&\quad+\|\z\|_{L^\infty}\|\p_x(|v|^2)\|_{H^{N_0-2}}\\
&\lesssim\|B\|_{H^{N_0}}\|v\|_{H^{N_0}}+\|\z\|_{W^{1,\infty}}\|v\|_{H^{N_0}}^2,
\end{aligned}\eeno
where we used the Sobolev inequality and the assumption $N_0\geq 4$.
Then we obtain
\beno\begin{aligned}
&\|\mathcal{C}_V\|_{H^{N_0}}
\lesssim\bigl(\|\z\|_{H^{N_0}}+\|v\|_{H^{N_0}}\bigr)\|B\|_{H^{N_0+1}}+\|\z\|_{W^{1,\infty}}\|v\|_{H^{N_0}}^2
\end{aligned}\eeno
Using \eqref{estimate for B} again, we have
\beno
\|B\|_{H^{N_0+1}}=\|B(\z,v)\|_{H^{N_0+1}}\lesssim\|\z\|_{L^\infty}\|v\|_{H^{N_0-1}}.
\eeno
Thus, we obtain
\beno
\|\mathcal{C}_V\|_{H^{N_0}}\lesssim\|\z\|_{W^{1,\infty}}\bigl(\|\z\|_{H^{N_0}}^2+\|v\|_{H^{N_0}}^2\bigr).
\eeno
This is \eqref{estimate for cubic term 2}.

\medskip

(6) {\it For the quartic term $\mathcal{N}_V$,} using \eqref{estimate for B}, we have
\beno
\|\mathcal{N}_V\|_{H^{N_0}}\lesssim\|(1+\p_x^2)\p_x B\|_{L^\infty}\|B\|_{H^{N_0}}+\|\p_x B\|_{L^\infty}\|B\|_{H^{N_0}}
\lesssim\|B\|_{H^{N_0}}^2.
\eeno
Using \eqref{estimate for B} again for $B=B(\z,v)$, we obtain
\beno
\|\mathcal{N}_V\|_{H^{N_0}}\lesssim\|\z\|_{L^\infty}^2\|v\|_{H^{N_0}}^2.
\eeno
This is \eqref{estimate for quartic term}. The proposition is proved.
\end{proof}

\setcounter{equation}{0}
\section{The proof of Theorem \ref{main theorem}}
In this section, we shall prove Theorem \ref{main theorem}. The proof relies on the continuity argument and the a priori estimates which are presented in the following subsections.
\subsection{Ansatz  for the continuity argument}
Our first ansatz for the continuity argument is about the amplitude of $\z$. We assume that
\beq\label{ansatz 2}
\|\z(t)\|_{L^\infty}\leq\f{1}{2C_B},\quad\text{for}\quad t\in[0,T_\e],
\eeq
where $C_B$ is a constant determined in Lemma \ref{lem for B}.
We define the energy functional
\beno
\mathcal{E}_{N_0}(t)=\|\z(t)\|_{H^{N_0}}^2+\|v(t)\|_{H^{N_0}}^2.
\eeno
Our second ansatz for the continuity argument is about the energy. We assume that
\beq\label{Ansatz}
\mathcal{E}_{N_0}(t)\leq 2C_0\e^2, \quad\text{for}\quad t\in[0,T_\e],
\eeq
where  $C_0>1$ is a universal constant that will be determined in the end of the proof. We take
\beno
T_\e=\f{C_1}{C_2}\e^{-\f43},\quad C_0=2C_1,
\eeno
where $C_1,C_2$ are constants stated in the following Proposition \ref{estimate Prop}.

We use the standard continuity argument: since for small $\e$,
\beno
\mathcal{E}_{N_0}(0)=\e^2<2C_0\e^2,\quad\|\z(0)\|_{L^\infty}\leq\f{1}{4C_B}<\f{1}{2C_B},
\eeno
the ansatz \eqref{ansatz 2} and \eqref{Ansatz} hold on a short time interval $[0,t^*)$, where $t^*$ is the maximal possible time on which  \eqref{ansatz 2} and \eqref{Ansatz} are correct. Without loss of generality, we assume that $T_\e=t^*$.
To close the continuity argument, we need the following two steps:

{\bf Step 1.} There exists a small constant $\e_0>0$, such that for all $\e<\e_0$, we can improve the ansatz \eqref{ansatz 2} to
\beq\label{improve ansatz 2}
\|\z(t)\|_{L^\infty}\leq\f{1}{4C_B},\quad\text{for}\quad t\in[0,T_\e].
\eeq

\smallskip

{\bf Step 2.} There exists a small constant $\e_0>0$, such that for all $\e<\e_0$, we can improve the ansatz \eqref{Ansatz} to
\beq\label{improve Ansatz}
\mathcal{E}_{N_0}(t)\leq C_0\e^2, \quad\text{for}\quad t\in[0,T_\e].
\eeq

Theorem \ref{main theorem} follows from the above two steps and the local regularity theorem. To complete the above two steps, we need  Proposition \ref{estimate Prop} in the following subsection. Thus, the rest of this section is concerned with the proof of Proposition \ref{estimate Prop}.

\subsection{The a priori energy estimates}
The main result of this section is about the a priori estimates of \eqref{Bsq}-\eqref{initial} which is stated in the following proposition.
\begin{proposition}\label{estimate Prop}
Under the ansatz \eqref{ansatz 2} and \eqref{Ansatz}, the solution $(\z,v)$ of \eqref{Bsq}-\eqref{initial} satisfies
\beq\label{priori estimate}
\mathcal{E}_{N_0}(t)\leq C_1\e^2+C_2t\e^{\f43}\cdot\e^2,\quad\text{for any } t\in(0,T_\e],
\eeq
where $C_1$ and $C_2$ are two universal constants, and $T_\e=\f{C_1}{C_2}\e^{-\f43}$.
\end{proposition}
\begin{proof}
We shall divide the proof into several steps.

{\bf Step 1. The a priori energy estimate.} Thanks to \eqref{estimate for B} and \eqref{ansatz 2}, we have
\beno
\|B(\z,\,v)\|_{H^{N_0}}\leq\f{1}{2}\|v\|_{H^{N_0}},
\eeno
which along with \eqref{good unknowns}, \eqref{good variable} implies that
\beq\label{equivalent energy functional}
\mathcal{E}_{N_0}(t)\sim\|\z(t)\|_{H^{N_0}}^2+\|u(t)\|_{H^{N_0}}^2\sim\|V(t)\|_{H^{N_0}}^2,\quad\text{for}\quad t\in[0,T_\e].
\eeq
By virtue of \eqref{equivalent energy functional}, we perform the energy estimate of \eqref{New formula}. First, we have
\beno\begin{aligned}
&\f12\f{d}{dt}\|V(t)\|_{H^{N_0}}^2=\text{Re}\{\bigl(\langle\p_x\rangle^{N_0}\mathcal{S}_V\,|\,\langle\p_x\rangle^{N_0}V\bigr)_2
+\bigl(\langle\p_x\rangle^{N_0}\mathcal{Q}_V\,|\,\langle\p_x\rangle^{N_0}V\bigr)_2
+\bigl(\langle\p_x\rangle^{N_0}\mathcal{R}_V\,|\,\langle\p_x\rangle^{N_0}V\bigr)_2\\
&\qquad+\bigl(\langle\p_x\rangle^{N_0}\mathcal{M}_V\,|\,\langle\p_x\rangle^{N_0}V\bigr)_2
+\bigl(\langle\p_x\rangle^{N_0}\mathcal{L}_V\,|\,\langle\p_x\rangle^{N_0}V\bigr)_2
+\bigl(\langle\p_x\rangle^{N_0}(\mathcal{C}_V+\mathcal{N}_V)\,|\,\langle\p_x\rangle^{N_0}V\bigr)_2\}.
\end{aligned}\eeno
Thanks to the estimates \eqref{estimate for cubic term 1}, \eqref{estimate for cubic term 2} and \eqref{estimate for quartic term} in Proposition \ref{New Bsq prop},  using \eqref{Ansatz} and \eqref{equivalent energy functional}, we obtain
\beq\label{priori 1}
\mathcal{E}_{N_0}(t)\lesssim\e^2+|\text{Re}(I+II+III+IV)|+t\e^4,
\eeq
where
\beq\label{quadratic  terms}\begin{aligned}
&I\eqdefa\sum_{\mu\in\{+,-\}}\int_0^t\int_{\R^2}\langle\xi\rangle^{2N_0}s_{\mu,+}(\xi,\eta)\widehat{V^\mu}(\xi-\eta)\widehat{V^+}(\eta)
\overline{\widehat{V^+}(\xi)}d\eta d\xi dt,\\
&II\eqdefa\sum_{\mu\in\{+,-\}}\int_0^t\int_{\R^2}\langle\xi\rangle^{2N_0}q_{\mu,-}(\xi,\eta)\widehat{V^\mu}(\xi-\eta)\widehat{V^-}(\eta)
\overline{\widehat{V^+}(\xi)}d\eta d\xi dt,\\
&III\eqdefa\sum_{\mu,\nu\in\{+,-\}}\int_0^t\int_{\R^2}\langle\xi\rangle^{2N_0}r_{\mu,\nu}(\xi,\eta)\widehat{V^\mu}(\xi-\eta)\widehat{V^\nu}(\eta)
\overline{\widehat{V^+}(\xi)}d\eta d\xi dt,\\
&IV\eqdefa\sum_{\mu,\nu\in\{+,-\}}\int_0^t\int_{\R^2}\langle\xi\rangle^{2N_0}m_{\mu,\nu}(\xi,\eta)\widehat{V^\mu}(\xi-\eta)\widehat{V^\nu}(\eta)
\overline{\widehat{V^+}(\xi)}d\eta d\xi dt.
\end{aligned}\eeq

{\bf Step 2. The evolution equation and estimates of the profile.}
To estimate the quadratic terms in \eqref{quadratic  terms}, we introduce the profiles $f$ and $g$ of $V$ and $\langle\p_x\rangle^{N_0}V$ as follows
\beno
f=e^{-it\Lambda}V\quad\text{and}\quad g=\langle\p_x\rangle^{N_0}f.
\eeno
Thanks to \eqref{equivalent energy functional}, we have
\beq\label{equivalent energy functional 2}
\mathcal{E}_{N_0}(t)\sim\|V(t)\|_{H^{N_0}}^2\sim\|f(t)\|_{H^{N_0}}^2=\|g(t)\|_{L^2}^2.
\eeq
By virtue of the definition of $f$ and the equation \eqref{New Bsq}, we have
\beq\label{evolution equation for profile}
\p_tf=e^{-it\Lambda}\Bigl(-\p_x(T_vV)+\f{i}{2}|\p_x|(T_{\z}V)+N_\z+i\f{\p_x}{|\p_x|}N_u\Bigr).
\eeq
Notice that the r.h.s of \eqref{evolution equation for profile} consists in  quadratic terms and higher order terms.

To bound $\p_tf$, we have to investigate the expressions of $N_\z$ and $N_u$. Thanks to \eqref{Bsq 17} and \eqref{Bsq 18}, there holds
\beno
\text{supp}\, \widehat{B(\cdot,\cdot)}(\xi)\subset\{\xi\in\R\,|\, |\xi|\geq 2^5\},
\eeno
which along with the expressions of $N_\z$ and $N_u$ shows that
\beno\begin{aligned}
&P_{\leq 0} N_\z=-\p_x P_{\leq 0}\bigl(\f12T_\z P_{\leq 5}v+R(\z,v)\bigr),\quad P_{\leq 0} N_u=\f{1}{2}\p_x P_{\leq 0}\bigl(T_\z P_{\leq 5}\z-R(v,v)\bigr).
\end{aligned}\eeno
Then we have
\beno
\|\f{1}{|\p_x|}P_{\leq 0}\p_tf\|_{L^2}\lesssim\bigl(\|v\|_{L^\infty}+\|\z\|_{L^\infty}\bigr)\bigl(\|v\|_{L^2}+\|\z\|_{L^2}+\|V\|_{L^2}\bigr).
\eeno

Whereas the expressions of $N_\z$ and $N_u$ give rise to
\beno\begin{aligned}
&\|\f{1}{|\p_x|}P_{\geq 1}\p_tf\|_{H^{N_0}}\lesssim\bigl(\|v\|_{L^\infty}+\|\z\|_{W^{3,\infty}}\bigr)
\bigl(\|V\|_{H^{N_0}}+\|v\|_{H^{N_0}}+\|\z\|_{H^{N_0}}+\|u\|_{H^{N_0}}\\
&\qquad+\|B(\z,v)\|_{H^{N_0}}\bigr)
+\|B(\p_t\z,v)\|_{H^{N_0-1}}+\|B\bigl(\z,\p_x(|v|^2)\bigr)\|_{H^{N_0-1}}.
\end{aligned}\eeno
The first equation of \eqref{Bsq} shows that
\beno
\p_t\z=-(1+\p_x^2)\p_xv-\p_x(\z v).
\eeno
Thanks to \eqref{estimate for B} and \eqref{equivalent energy functional}, we obtain
\beq\label{estimate for p_t f}
\|\f{1}{|\p_x|}\p_tf\|_{H^{N_0}}\lesssim\bigl(\|v\|_{W^{3,\infty}}+\|\z\|_{W^{3,\infty}}\bigr)\bigl(\|V\|_{H^{N_0}}+\|V\|_{H^{N_0}}^2\bigr)
\lesssim\e^2+\e^3\lesssim\e^2.
\eeq

\medskip

{\bf Step 3. Estimate for Re$(I)$.} In this step, we shall prove
\beq\label{estimate for I}
|\text{Re}(I)|\lesssim\e^2+t\e^{\f43}\cdot\e^2.
\eeq

By the expression of $I$, using \eqref{symbol s}, we have
\beno\begin{aligned}
&\bar{I}=-\sum_{\mu\in\{+,-\}}\int_0^t\int_{\R^2}\langle\xi\rangle^{2N_0}s_{\mu,+}(\xi,\eta)\overline{\widehat{V^\mu}(\xi-\eta)}
\,\overline{\widehat{V^+}(\eta)}
\widehat{V^+}(\xi)d\eta d\xi dt
\end{aligned}\eeno
Noticing that
\beno
\overline{\widehat{V^\mu}(\xi)}=\widehat{V^{-\mu}}(-\xi),
\eeno
we have
\beno\begin{aligned}
\bar{I}&=-\sum_{\mu\in\{+,-\}}\int_0^t\int_{\R^2}\langle\xi\rangle^{2N_0}s_{\mu,+}(\xi,\eta)\widehat{V^{-\mu}}(\eta-\xi)
\,\overline{\widehat{V^+}(\eta)}\widehat{V^+}(\xi)d\eta d\xi dt\\
&=-\sum_{\mu\in\{+,-\}}\int_0^t\int_{\R^2}\langle\eta\rangle^{2N_0}s_{\mu,+}(\eta,\xi)\widehat{V^{-\mu}}(\xi-\eta)
\widehat{V^+}(\eta)\overline{\widehat{V^+}(\xi)}d\eta d\xi dt.
\end{aligned}\eeno
Since Re$(I)=\f12(I+\bar{I})$, we obtain
\beno
\text{Re} (I)=\sum_{\mu\in\{+,-\}}\int_0^t\int_{\R^2}\tilde{s}_{\mu,+}(\xi,\eta)\widehat{V^{\mu}}(\xi-\eta)\cdot
\langle\eta\rangle^{N_0}\widehat{V^+}(\eta)\cdot\langle\xi\rangle^{N_0}\widehat{V^-}(-\xi)d\eta d\xi dt,
\eeno
where
\beno\begin{aligned}
\tilde{s}_{\mu,+}(\xi,\eta)&=\langle\xi\rangle^{-N_0}\langle\eta\rangle^{-N_0}\bigl(\langle\xi\rangle^{2N_0}s_{\mu,+}(\xi,\eta)
-\langle\eta\rangle^{2N_0}s_{-\mu,+}(\eta,\xi)\bigr).
\end{aligned}\eeno

Thanks to \eqref{bound of s},  we have
\beq\label{I 2}
\text{supp}\,\tilde{s}_{\mu,+}\subset\mathbb{S}\eqdefa\{(\xi,\eta)\in\R^2\,|\, |\xi-\eta|\leq 2^{-5}\max\{|\xi|,|\eta|\}\}.
\eeq
and
\beq\label{bound of s+}
|\tilde{s}_{\mu,+}(\xi,\eta)|\lesssim|\xi-\eta|\cdot1_{\mathbb{S}}(\xi,\eta).
\eeq

For simplicity, we denote
\beno
\mathfrak{S}_{\mu}(\xi,\eta)\eqdefa\tilde{s}_{\mu,+}(\xi,\eta)\widehat{V^{\mu}}(\xi-\eta)\cdot
\langle\eta\rangle^{N_0}\widehat{V^+}(\eta)\cdot\langle\xi\rangle^{N_0}\widehat{V^-}(-\xi).
\eeno
To estimate Re$(I)$, we rewrite $\mathfrak{S}_{\mu}(\xi,\eta)$ in terms of profiles $f^{\pm},g^{\pm}$ as follows
\beq\label{I 4}
\mathfrak{S}_{\mu}(\xi,\eta)=e^{it\Phi_{\mu,+}(\xi,\eta)}\tilde{s}_{\mu,+}(\xi,\eta)\widehat{f^{\mu}}(\xi-\eta)\cdot
\widehat{g^+}(\eta)\cdot\widehat{g^-}(-\xi).
\eeq
Thanks to Lemma \ref{phase lem 1},
we have
\beno
\Phi_{-,+}(\xi,\eta)=-\Phi_{+,+}(\eta,\xi).
\eeno
 Then the estimate of $\int_0^t\int_{\R^2}\mathfrak{S}_{-}(\xi,\eta)d\eta d\xi dt$ is similar to $\int_0^t\int_{\R^2}\mathfrak{S}_{+}(\xi,\eta)d\eta d\xi dt$. We only derive the estimate for $\int_0^t\int_{\R^2}\mathfrak{S}_{+}(\xi,\eta)d\eta d\xi dt$.

By the expression of $\Phi_{+,+}(\xi,\eta)$, we divide $\mathfrak{S}_{+}(\xi,\eta)$ into two parts as follows
\beno
\mathfrak{S}_{+}(\xi,\eta)=\underbrace{\mathfrak{S}_{+}(\xi,\eta)\cdot1_{(\xi-\eta)\cdot\eta>0}}_{\mathfrak{S}_{+}^{>0}(\xi,\eta)}
+\underbrace{\mathfrak{S}_{+}(\xi,\eta)\cdot1_{(\xi-\eta)\cdot\eta<0}}_{\mathfrak{S}_{+}^{<0}(\xi,\eta)}.
\eeno

\medskip

{\it Step 3.1. Estimate for the integral of $\mathfrak{S}_{+}^{>0}(\xi,\eta)$.} For $(\xi-\eta)\cdot\eta>0$, Lemma \ref{phase lem 1} shows that
\beno
\Phi_{+,+}(\xi,\eta)=3|\xi||\xi-\eta||\eta|.
\eeno
Now we split the integral of $\mathfrak{S}_{+}^{>0}(\xi,\eta)$ into two parts which  correspond to high and low frequencies respectively, i.e.,
\beno
|\xi|\geq 2^{-D-1}\quad\text{and}\quad |\xi|\leq 2^{-D},
\eeno
where $D\in\N$ is a large number that will be determined later on.

\smallskip

{\it (1). For $|\xi|\geq 2^{-D-1}$,} using \eqref{relation between xi and eta} and \eqref{bound of s+}, we have
\beq\label{I 5}
\Bigl|\f{\tilde{s}_{+,+}(\xi,\eta)}{i\Phi_{+,+}(\xi,\eta)}\Bigr|\lesssim\f{1}{|\xi||\eta|}\sim\f{1}{|\xi|^2}.
\eeq
Using \eqref{I 4}, we have
\beno
{\mathfrak{S}}_+^{> 0}(\xi,\eta)=\f{\tilde{s}_{+,+}(\xi,\eta)}{i\Phi_{+,+}(\xi,\eta)}\f{d}{dt}e^{it\Phi_{\mu,+}(\xi,\eta)}\widehat{f^+}(\xi-\eta)\cdot
\widehat{g^+}(\eta)\cdot\widehat{g^-}(-\xi)\cdot1_{(\xi-\eta)\cdot\eta>0}.
\eeno
Integrating by parts with respect to  t, we have
\beno\begin{aligned}
&\int_0^t\int_{\R^2}{\mathfrak{S}}_+^{>0}(\xi,\eta)\varphi_{\geq -D}(|\xi|)d\eta d\xi dt\\
&
=\underbrace{\int_{\R^2}\f{\tilde{s}_{+,+}(\xi,\eta)}{i\Phi_{+,+}(\xi,\eta)}e^{i\tau\Phi_{\mu,+}(\xi,\eta)}\widehat{f^+}(\tau,\xi-\eta)\cdot
\widehat{g^+}(\tau,\eta)\cdot\widehat{g^-}(\tau,-\xi)\cdot\varphi_{\geq -D}(|\xi|)\cdot1_{(\xi-\eta)\cdot\eta>0}d\eta d\xi \bigl|_{\tau=0}^t}_{A_1}\\
&\quad-\underbrace{\int_0^t\int_{\R^2} \f{\tilde{s}_{+,+}(\xi,\eta)}{i\Phi_{+,+}(\xi,\eta)}e^{i\tau\Phi_{\mu,+}(\xi,\eta)}\p_t\Bigl(\widehat{f^+}(\xi-\eta)\cdot
\widehat{g^+}(\eta)\cdot\widehat{g^-}(-\xi)\Bigr)\cdot\varphi_{\geq -D}(|\xi|)1_{(\xi-\eta)\cdot\eta>0}d\eta d\xi d\tau}_{A_2}.
\end{aligned}\eeno

Thanks to \eqref{relation between xi and eta} and \eqref{I 5},  we have
\beno\begin{aligned}
|A_1|&\lesssim \sum_{\tau\in\{0,t\}}\int_{\mathbb{S}}\f{1}{|\eta|^2}|\widehat{f}(\tau,\xi-\eta)|\cdot
|\widehat{g}(\tau,\eta)|\cdot|\widehat{g}(\tau,-\xi)|\varphi_{\geq-D}(|\xi|)d\eta d\xi \\
&\lesssim\sum_{\tau\in\{0,t\}}\|\widehat{f}(\tau,\xi)\|_{L^2}\|\f{1}{|\xi|^2}\widehat{g}(\tau,\xi)\|_{L^1(|\xi|\geq2^{-D-2})}\|\widehat{g}(\tau,\xi)\|_{L^2}\\
&\lesssim 2^{\f{3}{2}D}\bigl(\|f(0)\|_{L^2}\|g(0)\|_{L^2}^2+\|f(t)\|_{L^2}\|g(t)\|_{L^2}^2\bigr),
\end{aligned}\eeno
where we used the following formula in the last inequality
\beq\label{ineq 1}
\|\f{1}{|\xi|^r}\widehat{g}(\tau,\xi)\|_{L^1(|\xi|\geq 2^{-D-2})}\lesssim 2^{(r-\f12) D}\|\widehat{g}(\tau,\xi)\|_{L^2},\quad\hbox{for any } r>\f12.
\eeq
Whereas using \eqref{relation between xi and eta} and \eqref{I 5}, we have
\beno\begin{aligned}
|A_2|&\lesssim t\sup_{(0,t)}\int_{\mathbb{S}}\f{1}{|\xi|^2}\Bigl(|\p_t\widehat{f^+}(\xi-\eta)|\cdot
|\widehat{g^+}(\eta)|\cdot|\widehat{g^-}(-\xi)|+|\widehat{f^+}(\xi-\eta)|\cdot
|\p_t\bigl(\widehat{g^+}(\eta)\widehat{g^-}(-\xi)\bigr)|\Bigr)\varphi_{\geq -D}(|\xi|)d\eta d\xi\\
&\lesssim t\sup_{(0,t)}\bigl(\|\f{1}{|\xi|}\p_t\widehat{f}(\xi)\|_{L^2}\|\f{1}{|\xi|}\widehat{g}(\xi)\|_{L^1(|\xi|\geq 2^{-D-2})}\|\widehat{g}\|_{L^2}
+\|\widehat{f}\|_{L^2}\|\f{1}{|\xi|}\p_t\widehat{g}(\xi)\|_{L^2}\|\f{1}{|\xi|}\widehat{g}\|_{L^1(|\xi|\geq 2^{-D-2})}\bigr)\\
&\lesssim 2^{\f{1}{2}D}t\sup_{(0,t)}\bigl(\|\f{1}{|\p_x|}\p_tf\|_{L^2}\|g\|_{L^2}^2
+\|f\|_{L^2}\|\f{1}{|\p_x|}\p_tg\|_{L^2}\|g\|_{L^2}\bigr),
\end{aligned}\eeno
where we used \eqref{ineq 1} in the last inequality.

Thanks to  \eqref{Ansatz}, \eqref{equivalent energy functional 2} and \eqref{estimate for p_t f}, noticing that $g=\langle\p_x\rangle^{N_0}f$, we have
\beq\label{estimate for I>0 in high fre}
\bigl|\int_0^t\int_{\R^2}{\mathfrak{S}}_+^{>0}(\xi,\eta)\varphi_{\geq-D}(|\xi|)d\eta d\xi dt\bigl|
\lesssim2^{\f{3}{2}D}\e^3+2^{\f{1}{2}D}t\e^4.
\eeq

\smallskip

{\it (2). For $|\xi|< 2^{-D}$,} we have $|\eta|< 2^{-D+1}$ for any $(\xi,\eta)\in\mathbb{S}$. Using \eqref{bound of s+} and \eqref{I 4}, we have
\beno\begin{aligned}
&|\int_0^t\int_{\R^2}{\mathfrak{S}}_+^{>0}(\xi,\eta)\varphi_{\leq-D-1}(|\xi|)d\eta d\xi dt|\\
&
\lesssim t\sup_{(0,t)}\int_{-2^{-D}}^{2^{-D}}\int_{-2^{-D+1}}^{2^{-D+1}}|\xi-\eta||\widehat{f}(\xi-\eta)|\cdot
|\widehat{g}(\eta)|\cdot|\widehat{g}(\xi)|d\eta d\xi\\
&\lesssim t\sup_{(0,t)}\|\xi\widehat{f}(\xi)\|_{L^2}\|\widehat{g}(\xi)\|_{L^1(|\xi|<2^{-D+1})}\|\widehat{g}\|_{L^2}\lesssim 2^{-\f12D}t\sup_{(0,t)}\|\p_xf\|_{L^2}\|g\|_{L^2}^2,
\end{aligned}
\eeno
which along with \eqref{Ansatz} and \eqref{equivalent energy functional 2} implies that
\beq\label{estimate for I>0 in low fre}
\bigl|\int_0^t\int_{\R^2}{\mathfrak{S}}_+^{>0}(\xi,\eta)\varphi_{\leq-D-1}(|\xi|)d\eta d\xi dt\bigl|
\lesssim2^{-\f{1}{2}D}t\e^3.
\eeq

\smallskip

Taking $D=[\log_2\e^{-\f23}]$ (i.e, $2^{D}\sim \e^{-\f23}$) in \eqref{estimate for I>0 in high fre} and \eqref{estimate for I>0 in low fre}, we have
\beq\label{estimate for I>0}
\bigl|\int_0^t\int_{\R^2}{\mathfrak{S}}_+^{>0}(\xi,\eta)d\eta d\xi dt\bigl|
\lesssim\e^2+t\e^{\f{4}{3}}\cdot\e^2.
\eeq
Here the notation $[x]$ means the largest integer that does not exceed $x$.

\medskip

{\it Step 3.2. Estimate for the integral of $\mathfrak{S}_{+}^{<0}(\xi,\eta)$.} For $(\xi,\eta)\in\mathbb{S}$ with $(\xi-\eta)\cdot\eta<0$, there holds
\beq\label{I 1}
\xi\cdot\eta>0,\quad\f{31}{32}|\eta|\leq|\xi|<|\eta|.
\eeq
Lemma \ref{phase lem 1} yields
\beq\label{I 6}
\Phi_{+,+}(\xi,\eta)=-\f{1}{2}|\xi-\eta|\bigl(3|\xi|^2+3|\eta|^2+|\xi-\eta|^2-4\bigr)=-2|\xi-\eta|\phi_{+,+}(\xi,\eta),
\eeq
where
\beno
\phi_{+,+}(\xi,\eta)=\xi^2+\eta^2-\f12\xi\cdot\eta-1.
\eeno

Now, we split the frequencies space into three parts as follows:

\smallskip

{\it (1). For high frequencies $|\eta|\geq 2^5$ and low frequencies $|\eta|\leq\f{1}{2}$},
 using \eqref{bound of s}, we have
\beno\begin{aligned}
&|\Phi_{+,+}(\xi,\eta)|\sim|\xi-\eta||\eta|^2\quad\text{and}\quad |\f{\widetilde{s}_{+,+}(\xi,\eta)}{i\Phi_{+,+}(\xi,\eta)}|\lesssim\f{1}{|\eta|^2},
\quad\text{for}\quad |\eta|\geq 2^5,\\
&
|\Phi_{+,+}(\xi,\eta)|\sim|\xi-\eta|\quad\text{and}\quad |\f{\widetilde{s}_{+,+}(\xi,\eta)}{i\Phi_{+,+}(\xi,\eta)}|\lesssim1,
\quad\text{for}\quad |\eta|\leq\f12.
\end{aligned}\eeno

Similarly as in the derivation of  \eqref{estimate for I>0 in high fre}, integrating by parts with respect to t, we have
\beq\label{estimate for I<0 in low and high fre}
|\int_0^t\int_{\mathbb{S}}{\mathfrak{S}}_+^{< 0}(\xi,\eta)\cdot\bigl(\varphi_{\leq -2}(|\eta|)+\varphi_{\geq 6}(|\eta|)\bigr)d\eta d\xi dt|
\lesssim\e^3+t\e^4.
\eeq

\smallskip

{\it (2). For moderate frequencies with large modulation of $\phi_{+,+}(\xi,\eta)$,} i.e.,
\beno
|\eta|\in[\f14,\,2^6]\quad\text{and}\quad |\phi_{+,+}(\xi,\eta)|\geq 2^{-D-1},
\eeno
using \eqref{bound of s+} and \eqref{I 6}, we have
\beno
|\f{\widetilde{s}_{+,+}(\xi,\eta)}{i\Phi_{+,+}(\xi,\eta)}|\lesssim\f{1}{|\phi_{+,+}(\xi,\eta)|}\lesssim2^{D}.
\eeno
Following a similar argument as for \eqref{estimate for I>0 in high fre}, integrating by parts with respect to t, noticing that the integral set is bounded, we have
\beq\label{estimate for I<0 in med fre and high phase}
|\int_0^t\int_{\R^2}{\mathfrak{S}}_+^{< 0}(\xi,\eta)\cdot\varphi_{[-1,5]}(|\eta|)\varphi_{\geq -D}(\phi_{+,+}(\xi,\eta))d\eta d\xi dt|
\lesssim2^D\e^3+2^Dt\e^4.
\eeq

\smallskip

{\it (3). For moderate frequencies with small modulation of  $\phi_{+,+}(\xi,\eta)$,} i.e.,
\beno
|\eta|\in[\f14,\,2^6]\quad\text{and}\quad |\phi_{+,+}(\xi,\eta)|\leq 2^{-D},
\eeno
using \eqref{I 1}, we only consider the integral over the set
\beno
\mathbb{S}_+=\{(\xi,\eta)\in\R^2\,|\, \eta\in[\f14,2^6],\,\f{31}{32}\eta\leq\xi<\eta\},
\eeno
since the integral over the set
\beno
\mathbb{S}_-=\{(\xi,\eta)\in\R^2\,|\, \eta\in[-2^6,-\f14],\,\eta<\xi\leq\f{31}{32}\eta\},
\eeno
could be estimated in a similar way.

Introducing the coordinate transformation on $\mathbb{S}_+$ as follows
\beno\begin{aligned}
\Psi:\,\mathbb{S}_+&\rightarrow\widetilde{\mathbb{S}}_+\subset\R^2,\\
(\xi,\eta)&\mapsto(\tilde{\xi},\eta)=(\phi_{+,+}(\xi,\eta),\eta),
\end{aligned}\eeno
we have
\beq\label{I 7}
\det\bigl(\f{\p\Psi(\xi,\eta)}{\p(\xi,\eta)}\bigr)=\f{\p\phi_{+,+}(\xi,\eta)}{\p\xi}=2\xi-\f12\eta\sim\eta\sim 1,
\eeq
which implies that $\Psi$ is invertible and we denote by
\beno
(\xi,\eta)=\Psi^{-1}(\tilde\xi,\eta).
\eeno

Changing variables $(\xi,\eta)$ to $(\tilde\xi,\eta)$, using \eqref{I 7}, we have
\beno\begin{aligned}
&|\int_0^t\int_{\mathbb{S}_+}{\mathfrak{S}}_+^{< 0}(\xi,\eta)\cdot\varphi_{[-1,5]}(|\eta|)\varphi_{\leq -D-1}(\phi_{+,+}(\xi,\eta))d\eta d\xi dt|\\
&\lesssim t\sup_{(0,t)}\int_{\f14}^{2^6}\int_{-2^{-D}}^{2^{-D}}\bigl(|\widehat{f}(\xi-\eta)|\cdot|\widehat{g}(\xi)|\cdot1_{\mathbb{S}_+}(\xi,\eta)\bigr)|_{(\xi,\eta)=\Psi^{-1}(\tilde{\xi},\eta)}
\cdot
|\widehat{g}(\eta)|d\tilde\xi d\eta\\
&\lesssim 2^{-\f{1}{2}D}t\sup_{(0,t)}\|\widehat{g}(\eta)\|_{L^2}
\Bigl(\int_{\f14}^{2^6}\int_{-2^{-D}}^{2^{-D}}\bigl(|\widehat{f}(\xi-\eta)|^2\cdot|\widehat{g}(\xi)|^2\cdot1_{\mathbb{S}_+}(\xi,\eta)\bigr)
|_{(\xi,\eta)=\Psi^{-1}(\tilde{\xi},\eta)}d\tilde\xi d\eta\Bigr)^{\f12}.
\end{aligned}\eeno
Then changing variables $(\tilde\xi,\eta)$ to $(\xi,\eta)$, using \eqref{I 7},  we obtain
\beno
|\int_0^t\int_{\mathbb{S}_+}{\mathfrak{S}}_+^{< 0}(\xi,\eta)\cdot\varphi_{[-1,5]}(|\eta|)\varphi_{\leq -D-1}(\phi_{+,+}(\xi,\eta))d\eta d\xi dt|\lesssim2^{-\f{1}{2}D}t\sup_{(0,t)}\|f\|_{L^2}\|g\|_{L^2}^2,
\eeno
which along with \eqref{Ansatz} and \eqref{equivalent energy functional 2} implies that
\beq\label{estimate for I<0 in med fre and low phase}
|\int_0^t\int_{\mathbb{S}_+}{\mathfrak{S}}_+^{< 0}(\xi,\eta)\cdot\varphi_{[-1,5]}(|\eta|)\varphi_{\leq -D-1}(\phi_{+,+}(\xi,\eta))d\eta d\xi dt|
\lesssim 2^{-\f{1}{2}D}t\e^3.
\eeq
The same estimate holds for $|\int_0^t\int_{\mathbb{S}_-}{\mathfrak{S}}_+^{< 0}(\xi,\eta)\cdot\varphi_{[-1,5]}(|\eta|)\varphi_{\leq -D-1}(\phi_{+,+}(\xi,\eta))d\eta d\xi dt|$.

Taking $D=[\log_2 \e^{-\f{2}{3}}]$ (i.e., $2^D\sim\e^{-\f{2}{3}}$) in \eqref{estimate for I<0 in med fre and high phase} and \eqref{estimate for I<0 in med fre and low phase},  we obtain
\beq\label{estimate for I<0 in med fre}
|\int_0^t\int_{\R^2}{\mathfrak{S}}_+^{< 0}(\xi,\eta)\cdot\varphi_{[-1,5]}(|\eta|)d\eta d\xi dt|
\lesssim\e^2+t\e^{\f43}\cdot \e^2.
\eeq

\smallskip

Thanks to \eqref{estimate for I<0 in low and high fre} and \eqref{estimate for I<0 in med fre}, we obtain
\beq\label{estimate for I<0}
|\int_0^t\int_{\R^2}{\mathfrak{S}}_+^{< 0}(\xi,\eta)d\eta d\xi dt|
\lesssim\e^2+t\e^{\f43}\cdot\e^2.
\eeq

\medskip

{\it Step 3.3. Estimate for Re$(I)$.} Combining \eqref{estimate for I>0} and \eqref{estimate for I<0}, we get
\beno
|\int_0^t\int_{\R^2}{\mathfrak{S}}_+(\xi,\eta)d\eta d\xi dt|
\lesssim\e^2+t\e^{\f43}\cdot\e^2.
\eeno
The same estimate holds for $\int_0^t\int_{\R^2}{\mathfrak{S}}_-(\xi,\eta)d\eta d\xi dt$.
Then we obtain
\beno
|\text{Re} (I)|\leq|\int_0^t\int_{\R^2}{\mathfrak{S}}_+(\xi,\eta)d\eta d\xi dt|+|\int_0^t\int_{\R^2}{\mathfrak{S}}_-(\xi,\eta)d\eta d\xi dt|\lesssim\e^2+t\e^{\f43}\cdot\e^2,
\eeno
This is exactly \eqref{estimate for I}.

\medskip

{\bf Step 4. Estimate for Re$(II)$.} In this step, we will prove
\beq\label{estimate for II}
|\text{Re} (II)|\lesssim\e^2+t\e^{\f43}\cdot\e^2,
\eeq

By the expression of $II$, denoting by
\beno\begin{aligned}
\mathfrak{Q}_{\mu,-}(\xi,\eta)\eqdefa\tilde{q}_{\mu,-}(\xi,\eta)\cdot\widehat{V^\mu}(\xi-\eta)
\cdot\langle\eta\rangle^{N_0}\widehat{V^-}(\eta)
\cdot\langle\xi\rangle^{N_0}\overline{\widehat{V^+}(\xi)},
\end{aligned}\eeno
with
\beno
\tilde{q}_{\mu,-}(\xi,\eta)=\langle\eta\rangle^{-N_0}\langle\xi\rangle^{N_0}q_{\mu,-}(\xi,\eta),
\eeno
we have
\beno
II=\sum_{\mu\in\{+,-\}}\int_0^t\int_{\R^2}\mathfrak{Q}_{\mu,-}(\xi,\eta)d\eta d\xi dt.
\eeno

{\it Step 4.1. Estimate for $\sum_{\mu\in\{+,-\}}\int_0^t\int_{\R^2}\mathfrak{Q}_{\mu,-}(\xi,\eta)d\eta d\xi dt$.} Now, we rewrite $\mathfrak{Q}_{\mu,-}(\xi,\eta)$ in terms of profiles  as follows
\beno
 \mathfrak{Q}_{\mu,-}(\xi,\eta)=e^{it\Phi_{\mu,-}(\xi,\eta)}\widetilde{q}_{\mu,-}(\xi,\eta)
 \widehat{f^\mu}(\xi-\eta)\widehat{g^-}(\eta)\widehat{g^-}(-\xi).
\eeno

Thanks to \eqref{bound of q}, we have
\beq\label{I 17}
|\widetilde{q}_{\mu,-}(\xi,\eta)|\lesssim|\xi|\cdot \varphi_{\leq 5}(|\eta|)\cdot\varphi_{\leq -6}\Bigl(\f{|\xi-\eta|}{|\eta|}\Bigr).
\eeq
Lemma \ref{phase lem 1} and the fact $\xi\cdot\eta>0$ (in \eqref{relation between xi and eta}) yield
\beno
\Phi_{+,-}(\xi,\eta)=
\left\{\begin{aligned}
&\f12|\eta|\bigl(3|\xi-\eta|^2+3|\xi|^2+|\eta|^2-4\bigr)\quad\text{if}\quad |\xi|>|\eta|,\\
&\f12|\xi|\bigl(3|\xi-\eta|^2+3|\eta|^2+|\xi|^2-4\bigr),
\quad\text{if}\quad |\xi|<|\eta|.
\end{aligned}\right.
\eeno
Then there hold
\beq\label{I 15}
\Phi_{-,-}(\xi,\eta)=\Phi_{+,-}(\eta,\xi), \quad\text{and}\quad\Phi_{+,-}(\xi,\eta)=\Phi_{+,-}(\eta,\xi).
\eeq
Due to \eqref{I 15}, we only need to estimate the integral of $\mathfrak{Q}_{+,-}(\xi,\eta)$ over the set with restriction $|\xi|>|\eta|$.

{\it For $|\xi|>|\eta|$}, we have
\beno
\Phi_{+,-}(\xi,\eta)=2|\eta|\phi_{+,-}(\xi,\eta)\quad\text{with}\quad \phi_{+,-}(\xi,\eta)=\eta^2-\f32\xi\eta+\f32\xi^2-1.
\eeno

A similar argument as in Step 3.2 leads to
\beq\label{estimate for II- in low fre}
|\int_0^t\int_{\R^2}\mathfrak{Q}_{+,-}(\xi,\eta)1_{|\xi|>|\eta|}\varphi_{\leq -2}(|\eta|)d\eta d\xi dt|\lesssim\e^3+t\e^4,
\eeq
\beq\label{estimate for II- in med fre and high phase}
|\int_0^t\int_{\R^2}\mathfrak{Q}_{+,-}(\xi,\eta)1_{|\xi|>|\eta|}\varphi_{[-1,5]}(|\eta|)\varphi_{\geq-D}(\phi_{+,-}(\xi,\eta))d\eta d\xi dt|\lesssim 2^D\e^3+2^Dt\e^4,
\eeq
\beq\label{estimate for II- in med fre and low phase}
|\int_0^t\int_{\R^2}\mathfrak{Q}_{+,-}(\xi,\eta)1_{|\xi|>|\eta|}\varphi_{[-1,5]}(|\eta|)\varphi_{\leq-D-1}(\phi_{+,-}(\xi,\eta))d\eta d\xi dt|\lesssim 2^{-\f12D}t\e^3,
\eeq
where $D\in\N$ need to be determined later on. Here we only verify \eqref{estimate for II- in med fre and low phase}. Indeed, since $|\xi-\eta|\leq 2^{-5}|\eta|$, we only consider the integral over set
\beno
\mathbb{S}_{>}=\{(\xi,\eta)\in\R^2\,|\,\eta\in[\f{1}{4},2^6],\,\eta<\xi\leq\f{33}{32}\eta\},
\eeno
since the same estimate will also hold for the integral over set
\beno
\mathbb{S}_{<}=\{(\xi,\eta)\in\R^2\,|\,\eta\in[-2^6,-\f{1}{4}],\,\f{33}{32}\eta\leq\xi<\eta\}.
\eeno
Introducing the coordinates transformation
on $\mathbb{S}_{>}$ as follows
\beno\begin{aligned}
\Psi_>:\,\mathbb{S}_{>}&\rightarrow \widetilde{\mathbb{S}}_{>}\subset\R^2,\\
(\xi,\eta)&\mapsto(\xi,\tilde{\eta})=(\xi,\phi_{+,-}(\xi,\eta)),
\end{aligned}\eeno
we have
\beq\label{I 8}
\det\bigl(\f{\p\Psi_>(\xi,\eta)}{\p(\xi,\eta)}\bigr)=\f{\p\phi_{+,-}(\eta)}{\p\eta}=2\eta-\f32\xi\sim\eta\sim 1,
\eeq
which implies that $\Psi_>$ is invertible. With \eqref{I 8}, following the similar derivation of \eqref{estimate for I<0 in med fre and low phase}, we obtain \eqref{estimate for II- in med fre and low phase}.

Taking $D=[\log_2 \e^{-\f23}]$ (i.e., $2^D\sim\e^{-\f23}$) in \eqref{estimate for II- in med fre and high phase} and \eqref{estimate for II- in med fre and low phase}, using \eqref{estimate for II- in low fre}, \eqref{estimate for II- in med fre and high phase} and \eqref{estimate for II- in med fre and low phase}, we get
\beno
|\int_0^t\int_{\R^2}\mathfrak{Q}_{+,-}(\xi,\eta)\cdot1_{(|\xi|>|\eta|)}d\eta d\xi dt|\lesssim\e^2+t\e^{\f43}\cdot \e^2.
\eeno
The same estimate hold for $\int_0^t\int_{\R^2}\mathfrak{Q}_{+,-}(\xi,\eta)\cdot1_{(|\xi|<|\eta|)}d\eta d\xi dt$ and $\int_0^t\int_{\R^2}\mathfrak{Q}_{-,-}(\xi,\eta)d\eta d\xi dt$. We finally obtain
\beno
|\sum_{\mu\in\{+,-\}}\int_0^t\int_{\R^2}\mathfrak{Q}_{\mu,-}(\xi,\eta)d\eta d\xi dt|\lesssim\e^2+t\e^{\f43}\cdot \e^2.
\eeno
This is \eqref{estimate for II}.

\medskip
{\bf Step 5. Estimate for Re$(III)$.} Firstly, we rewrite $III$ in terms of the profiles as follows
 \beno
 III=\sum_{\mu,\nu\in\{+,-\}}\int_0^t\int_{\R^2}e^{it\Phi_{\mu,\nu}(\xi,\eta)}\widetilde{r}_{\mu,\nu}(\xi,\eta)
 \widehat{f^\mu}(\xi-\eta)\widehat{g^\nu}(\eta)\widehat{g^-}(-\xi)d\eta d\xi dt,
\eeno
where
 \beno
 \widetilde{r}_{\mu,\nu}(\xi,\eta)=\langle\xi\rangle^{N_0}\langle\eta\rangle^{-N_0}r_{\mu,\nu}(\xi,\eta).
 \eeno

Thanks to \eqref{bound of r}, we have
\beno
|\widetilde{r}_{\mu,\nu}(\xi,\eta)|\lesssim|\xi-\eta|\cdot\varphi_{\geq6}(|\eta|)\cdot\varphi_{\leq-6}\Bigl(\f{|\xi-\eta|}{|\eta|}\Bigr),
\eeno
where we used $|\xi|\sim|\eta|$ which is stated in \eqref{relation between xi and eta}.
After similar derivations as  in Step 3 and Step 4, we obtain
\beq\label{estimate for III}
\bigl|\text{Re}(III)\bigl|
\lesssim\e^2+t\e^{\f43}\cdot\e^2.
\eeq

\medskip

{\bf Step 6. Estimate for Re$(IV)$.} In this step, we shall prove
\beq\label{estimate for IV}
|\text{Re}(IV)|\lesssim\e^2+t\e^{\f43}\cdot\e^2.
\eeq

First, denoting by
\beno
\mathfrak{M}_{\mu,\nu}(\xi,\eta)\eqdefa\langle\xi\rangle^{2N_0}m_{\mu,\nu}(\xi,\eta)\widehat{V^\mu}(\xi-\eta)\widehat{V^\nu}(\eta)
\overline{\widehat{V^+}(\xi)},
\eeno
we have
\beno
IV=\sum_{\mu,\nu\in\{+,-\}}\int_0^t\int_{\R^2}\mathfrak{M}_{\mu,\nu}(\xi,\eta)d\eta d\xi dt.
\eeno
Thanks to \eqref{bound of m}, we have
\beno
\text{supp}\,\mathfrak{M}_{\mu,\nu}\subset\{(\xi,\eta)\in\R^2\,|\,2^{-7}|\eta|\leq|\xi-\eta|\leq2^8|\eta|\},
\eeno
i.e., for any $(\xi,\eta)\in$supp$\,\mathfrak{M}_{\mu,\nu}$,
\beq\label{Relation between xi and eta in M}
|\xi-\eta|\sim|\eta|, \quad |\xi|\lesssim|\eta|.
\eeq
By the definitions of the profiles, we rewrite $\mathfrak{M}_{\mu,\nu}(\xi,\eta)$ to
\beno
\mathfrak{M}_{\mu,\nu}(\xi,\eta)=e^{it\Phi_{\mu,\nu}(\xi,\eta)}\widetilde{m}_{\mu,\nu}(\xi,\eta)\widehat{f^\mu}(\xi-\eta)\widehat{g^\nu}(\eta)
\overline{\widehat{g^+}(\xi)},
\eeno
where
\beno
\widetilde{m}_{\mu,\nu}(\xi,\eta)=\langle\xi\rangle^{N_0}\langle\eta\rangle^{-N_0}m_{\mu,\nu}(\xi,\eta).
\eeno

Due to \eqref{bound of m} and \eqref{Relation between xi and eta in M}, we have
\beq\label{I 19}
|\widetilde{m}_{\mu,\nu}(\xi,\eta)|\lesssim|\xi|.
\eeq

Thanks to Lemma \ref{phase lem 1}, we shall only derive the estimates for the integral of $\mathfrak{M}_{+,+}(\xi,\eta)$.
\medskip

{\it Step 6.1. The integral over the set with $(\xi-\eta)\cdot\eta>0$.} For $(\xi-\eta)\cdot\eta>0$, Lemma \ref{phase lem 1} yields
\beno
\Phi_{+,+}(\xi,\eta)=3|\xi||\xi-\eta||\eta|,
\eeno
which along with \eqref{I 19} shows
\beno
\bigl|\f{\widetilde{m}_{+,+}(\xi,\eta)}{\Phi_{+,+}(\xi,\eta)}\bigr|\lesssim\f{1}{|\xi-\eta||\eta|}.
\eeno
Similarly as the derivation of \eqref{estimate for I>0}, using \eqref{Relation between xi and eta in M}, we have
\beq\label{estimate for IV+>0}
|\int_0^t\int_{\R^2}\mathfrak{M}_{+,+}(\xi,\eta)\cdot 1_{(\xi-\eta)\cdot\eta>0}d\eta d\xi|\lesssim\e^2+t\e^{\f43}\cdot\e^2.
\eeq

\smallskip

{\it Step 6.2. The integral over the set with $(\xi-\eta)\cdot\eta<0$.} For $(\xi-\eta)\cdot\eta<0$,  Lemma \ref{phase lem 1} shows
\beq\label{I 20}\begin{aligned}
&\Phi_{+,+}(\xi,\eta)=-\f12\min\{|\xi-\eta|,|\eta|\}\phi_{+,+}(\xi,\eta)\\
\text{with}\quad
&\phi_{+,+}(\xi,\eta)=3|\xi|^2+3\max\{|\xi-\eta|^2,|\eta|^2\}+\min\{|\xi-\eta|^2,|\eta|^2\}-4.
\end{aligned}\eeq

Now we split the frequency space into three parts as follows:

{\it (1). For high frequencies $|\eta|>4$ and low frequency $|\eta|<2^{-9}$ }, using the fact that $|\xi-\eta|\in[2^{-7}|\eta|,\,2^8|\eta|]$ and \eqref{I 19}, we have
\beno\begin{aligned}
&|\Phi_{+,+}(\xi,\eta)|\sim |\eta|^3\quad\text{and}\quad \bigl|\f{\widetilde{m}_{+,+}(\xi,\eta)}{\Phi_{+,+}(\xi,\eta)}\bigr|\lesssim\f{1}{|\eta|^2},\quad\text{if}\quad |\eta|>4,\\
&|\Phi_{+,+}(\xi,\eta)|\sim|\eta|\quad\text{and}\quad \bigl|\f{\widetilde{m}_{+,+}(\xi,\eta)}{\Phi_{+,+}(\xi,\eta)}\bigr|\lesssim1,\quad\text{if}\quad |\eta|<2^{-9}.
\end{aligned}\eeno

Following similar derivation as \eqref{estimate for I>0 in high fre}, using \eqref{Relation between xi and eta in M}, we have
\beq\label{estimate for IV+<0 in high and low fre}
|\int_0^t\int_{\R^2}\mathfrak{M}_{+,+}(\xi,\eta)\cdot 1_{(\xi-\eta)\cdot\eta<0}\cdot\bigl(\varphi_{\leq -10}(|\eta|)+\varphi_{\geq3}(|\eta|)\bigr)d\eta d\xi dt|\lesssim\e^3+t\e^4.
\eeq

\smallskip

{\it (2). For moderate frequencies with large modulation of $\phi_{+,+}(\xi,\eta)$,} i.e.,
\beno
 |\eta|\in[2^{-10},8]\quad\text{and}\quad |\phi_{+,+}(\xi,\eta)|\geq 2^{-D-1},
\eeno
following similar derivation as \eqref{estimate for I<0 in med fre and high phase}, we have
\beq\label{estimate for IV+<0 in med fre and high phase 1}
|\int_0^t\int_{\mathbb{S}'_1}\mathfrak{M}_{+,+}(\xi,\eta)\cdot 1_{(\xi-\eta)\cdot\eta<0}\cdot\varphi_{[-9,2]}(|\eta|)\cdot\varphi_{\geq-D}(\phi_{+,+}(\xi,\eta))d\eta d\xi dt|\lesssim2^D\e^3+2^Dt\e^4.
\eeq

\smallskip

{\it (3). For moderate frequencies with small modulation of $\phi_{+,+}(\xi,\eta)$,} i.e.,
\beno
 |\eta|\in[2^{-10},8]\quad\text{and}\quad |\phi_{+,+}(\xi,\eta)|\leq 2^{-D},
\eeno
we divide the integral set
\beno
\mathbb{S}'\eqdefa\{(\xi,\eta)\in\R^2\,|\, (\xi-\eta)\cdot\eta<0,\,|\eta|\in[2^{-10},8],\,\,|\xi-\eta|\in[2^{-7}|\eta|,2^8|\eta|]\}
\eeno
into two sets as follows
\beno
\mathbb{S}'=\underbrace{\{(\xi,\eta)\in\mathbb{S}'\,|\, \xi\cdot\eta>0,\,\,|\eta|>|\xi|\}}_{\mathbb{S}'_1}\cup
\underbrace{\{(\xi,\eta)\in\mathbb{S}'\,|\, \xi\cdot\eta<0\}}_{\mathbb{S}'_2}.
\eeno

{\it (i). When $(\xi,\eta)\in\mathbb{S}'_1$,} we have $|\xi-\eta|=|\eta|-|\xi|$ .
Due to \eqref{I 20}, there holds
\beno
\phi_{+,+}(\xi,\eta)=4\xi^2+4\eta^2-2\xi\cdot\eta-4.
\eeno

Following similar derivation as  \eqref{estimate for I<0 in med fre and low phase}, we have
\beq\label{estimate for IV+<0 in med fre and low phase 1}
|\int_0^t\int_{\mathbb{S}'_1}\mathfrak{M}_{+,+}(\xi,\eta)\cdot 1_{(\xi-\eta)\cdot\eta<0}\cdot\varphi_{[-9,2]}(|\eta|)\cdot\varphi_{\leq-D-1}(\phi_{+,+}(\xi,\eta))d\eta d\xi dt|\lesssim2^{-\f12D}t\e^3,
\eeq
Here we only need to verify \eqref{estimate for IV+<0 in med fre and low phase 1} on set
\beno
\mathbb{S}'_{1+}=\{(\xi,\eta)\in\mathbb{S}'_1\,|\, \eta>\xi>0,\,\,\eta\in[2^{-10},8]\}.
\eeno
According to the proof of \eqref{estimate for I<0 in med fre and low phase}, it is reduced to check that  there exists an invertible coordinates transformation on $\mathbb{S}'_{1+}$. Indeed, introducing the coordinates transformation on $\mathbb{S}'_{1+}$ as follows
\beno\begin{aligned}
\Psi_{1+}:\,\mathbb{S}'_{1+}&\rightarrow \widetilde{\mathbb{S}'}_{1+}\subset\R^2,\\
(\xi,\eta)&\mapsto(\xi,\tilde{\eta})=(\xi,\phi_{+,+}(\xi,\eta)),
\end{aligned}\eeno
we have
\beno
\det\bigl(\f{\p\Psi_{1+}(\xi,\eta)}{\p(\xi,\eta)}\bigr)=\f{\p\phi_{+,+}(\xi,\eta)}{\p\eta}=8\eta-2\xi\sim\eta\sim 1,
\eeno
which implies that $\Psi_{1+}$ is invertible.

{\it (ii). When $(\xi,\eta)\in\mathbb{S}'_2$,} we have $|\xi-\eta|=|\xi|+|\eta|$. Due to \eqref{I 20}, there holds
\beno
\phi_{+,+}(\xi,\eta)=6\xi^2+4\eta^2-6\xi\cdot\eta-4.
\eeno
Similarly as \eqref{estimate for IV+<0 in med fre and low phase 1}, we have
\beq\label{estimate for IV+<0 in med fre and low phase 2}
|\int_0^t\int_{\mathbb{S}'_2}\mathfrak{M}_{+,+}(\xi,\eta)\cdot 1_{(\xi-\eta)\cdot\eta<0}\cdot\varphi_{[-9,2]}(|\eta|)\cdot\varphi_{\leq-D-1}(\phi_{+,+}(\xi,\eta))d\eta d\xi dt|\lesssim2^{-\f12D}t\e^3.
\eeq
Here we only need to check that there exists invertible coordinates transformation on $\mathbb{S}'_2$. Since $\xi\cdot\eta<0$ for any $(\xi,\eta)\in\mathbb{S}'_2$, we only consider the set
\beno
\mathbb{S}'_{2>}=\{(\xi,\eta)\in\mathbb{S}'_2\,|\,\xi<0,\,\eta\in[2^{-10},8]\}
\eeno
Introducing coordinates transformation on $\mathbb{S}'_{2>}$ as follows
\beno\begin{aligned}
\Psi_{2>}:\,\mathbb{S}'_{2>}&\rightarrow \widetilde{\mathbb{S}'}_{2>}\subset\R^2,\\
(\xi,\eta)&\mapsto(\xi,\tilde{\eta})=(\xi,\phi_{+,+}(\xi,\eta)),
\end{aligned}\eeno
we have
\beno
\det\bigl(\f{\p\Psi_{2>}(\xi,\eta)}{\p(\xi,\eta)}\bigr)=\f{\p\phi_{+,+}(\xi,\eta)}{\p\eta}=8\eta-6\xi.
\eeno
Since $|\xi-\eta|=|\xi|+|\eta|\in[2^{-7}|\eta|,2^8|\eta|]$, we have
\beno
\xi\in[-(2^8-1)\eta,0),
\eeno
which along with the fact $\eta\in[2^{-10},8]$ implies
\beno
\det\bigl(\f{\p\Psi_{2>}(\xi,\eta)}{\p(\xi,\eta)}\bigr)\sim\eta\sim1.
\eeno
Then $\Psi_{2>}$ is invertible.

Taking  $D=[\log_2 \e^{-\f23}]$ (i.e., $2^D\sim\e^{-\f23}$) in \eqref{estimate for IV+<0 in med fre and high phase 1},
\eqref{estimate for IV+<0 in med fre and low phase 1} and \eqref{estimate for IV+<0 in med fre and low phase 2}, we obtain
\beq\label{estimate for IV+<0 in med fre}
|\int_0^t\int_{\R^2}\mathfrak{M}_{+,+}(\xi,\eta)\cdot 1_{(\xi-\eta)\cdot\eta<0}\cdot\varphi_{[-9,2]}(|\eta|)d\eta d\xi dt|\lesssim\e^2+t\e^{\f43}\cdot\e^2,
\eeq

\medskip

Thanks to \eqref{estimate for IV+>0}, \eqref{estimate for IV+<0 in high and low fre} and \eqref{estimate for IV+<0 in med fre}, we obtain
\beq\label{estimate for IV+}
|\int_0^t\int_{\R^2}\mathfrak{M}_{+,+}(\xi,\eta)d\eta d\xi dt|\lesssim\e^2+t\e^{\f43}\cdot\e^2.
\eeq
The same estimate holds for $\int_0^t\int_{\R^2}\mathfrak{M}_{\mu,\nu}(\xi,\eta)d\eta d\xi dt$. Then we obtain \eqref{estimate for IV}.

\medskip

Combining \eqref{priori 1}, \eqref{estimate for I}, \eqref{estimate for II}, \eqref{estimate for III} and \eqref{estimate for IV}, we finally obtain \eqref{priori estimate}. The Proposition is proved.
\end{proof}

\medskip

\setcounter{equation}{0}
\section{The proof of Theorem \ref{long time existence thm}}
In this section, we shall sketch the proof of Theorem \ref{long time existence thm}. Since the small parameter $\epsilon$ is considered in \eqref{Bsqeps}, we have to modify the proof of Theorem \ref{main theorem} slightly.

\subsection{Symmetrization of \eqref{Bsqeps}}
Similarly as the derivation of \eqref{New Bsq}, we firstly introduce good unknowns $(\z,u)$ with
\beq\label{good unknowns for bsq e}
u=v+\epsilon B^\epsilon(\z,v),
\eeq
where $B^\epsilon(\cdot,\cdot)$ is a bilinear operator defined as
\beno
B^\epsilon(f,g)=\f12T_f\bigl((1+\epsilon\p_x^2)^{-1}\varphi_{\geq6}(\sqrt\epsilon|\p_x|)g\bigr).
\eeno
Without confusion, we sometimes use $B^\epsilon$ to denote the bilinear term $B^\epsilon(\z,v)$. Defining
\beq\label{good variable for bsq e}
V=\z+i\f{\p_x}{|\p_x|}u,
\eeq
we get
\beq\label{New Bsq e}
\p_tV-i\Lambda_\epsilon V+\epsilon\p_x(T_vV)-\f{i}{2}\epsilon|\p_x|(T_\z V)=N_\z^\epsilon+i\f{\p_x}{|\p_x|}N_u^\epsilon,
\eeq
where $\Lambda_\epsilon=|\p_x|(1-\epsilon|\p_x|^2)$ and
\beno\begin{aligned}
&N_\z^\epsilon=-\f{\epsilon}{2}\p_x\bigl(T_\z \varphi_{\leq5}(\sqrt\epsilon|\p_x|)u\bigr)-\f{\epsilon^2}{2}\p_x\bigl(T_\z \varphi_{\geq6}(\sqrt\epsilon|\p_x|)B^\epsilon\bigr)+\epsilon^2\p_x(T_\z B^\epsilon)\\
&\qquad+\f{\epsilon^2}{2}\p_x\Bigl([\p_x^2,T_\z](1+\epsilon\p_x^2)^{-1}\varphi_{\geq6}(\sqrt\epsilon|\p_x|)v\Bigr)-\epsilon\p_x\bigl(R(\z,v)\bigr),\\
&N_u^\epsilon=\f{\epsilon}{2}\p_x\bigl(T_\z\varphi_{\leq5}(\sqrt\epsilon|\p_x|)\z\bigr)+\f{\epsilon}{2}T_{\p_x\z}\varphi_{\geq6}(\sqrt\epsilon|\p_x|)\z
+\epsilon^2\p_x(T_vB^\epsilon)
-\f{\epsilon}{2}\p_x\bigl(R(v,v)\bigr)\\
&\qquad
+\epsilon B^\epsilon(\p_t\z,v)-\f{\epsilon^2}{2}B^\epsilon\bigl(\z,\p_x(|v|^2)\bigr),
\end{aligned}\eeno
where we used \eqref{commutator} and the definition of $B^\epsilon(\cdot,\cdot)$. Here we used the Fourier multipliers $\varphi_{\leq k}(\cdot)$, $\varphi_{\geq k}(\cdot)$ and $\varphi_k(\cdot)$, instead of their Littlewood-Paley projection operators $P_{\leq k}$, $P_{\geq k}$ and $P_k$, respectively (see subsection 2.2.).

Following the proof of Lemma \ref{lem for B},  for any $f\in L^\infty(\R)$ and $g\in H^s(\R)$ with $s\geq-2$, we have
\beq\label{estimate for B e 1}
\mathcal{F}\bigl(B^\epsilon(f,g)\bigr)(\xi)=\overline{\mathcal{F}\bigl(B^\epsilon(f,g)\bigr)(-\xi)}
\eeq
and
\beq\label{estimate for B e}
\|B^\epsilon(f,g)\|_{H^{s+k}}\leq C_{B^\epsilon}\epsilon^{-\f{k}{2}}\|f\|_{L^\infty}\|g\|_{H^s},\quad\text{for}\quad k=0,1,2,
\eeq
where $C_{B^\epsilon}>0$ is a universal constant.

\subsection{Main proposition on the symmetric system \eqref{New Bsq e}} For \eqref{New Bsq e}, arranging the quadratic terms in terms
of $V^+$ and $V^-$, we have a  proposition similar to  Proposition \ref{New Bsq prop}.
\begin{proposition}\label{New Bsq e prop}
Assume that $(\z,v)\in H^{N_0}(\R)$ with $N_0\geq 4$ solves \eqref{Bsqeps}. Then $V$ defined in \eqref{good variable for bsq e} satisfies the following system
\beq\label{New formula for bsq e}
\p_tV-i\Lambda_\epsilon V=\mathcal{S}^\epsilon_V+\mathcal{Q}^\epsilon_V+\mathcal{L}^\epsilon_V+\mathcal{N}^\epsilon_V,
\eeq
where
\begin{itemize}
\item The quadratic term $\mathcal{S}^\epsilon_V$ is of the form
\beno
\mathcal{S}^\epsilon_V=S_{+,+}^\epsilon(V^+,V^+)+S^\epsilon_{-,+}(V^-,V^+).
\eeno
And the symbol $s^\epsilon_{\mu,+}(\xi,\eta)$ of $S^\epsilon_{\mu,+}$ (for $\mu=+,-$) satisfies
\beq\label{symbol s e}
\overline{s^\epsilon_{\mu,+}(\xi,\eta)}=-s^\epsilon_{\mu,+}(\xi,\eta),
\eeq
\beq\label{bound of s e}
|\langle\xi\rangle^{-N_0}\langle\eta\rangle^{-N_0}\bigl(\langle\xi\rangle^{2N_0}s^\epsilon_{\mu,+}(\xi,\eta)
-\langle\eta\rangle^{2N_0}s^\epsilon_{-\mu,+}(\eta,\xi)\bigr)|\lesssim\epsilon|\xi-\eta|\cdot\varphi_{\leq -6}\Bigl(\f{|\xi-\eta|}{\max\{|\xi|,|\eta|\}}\Bigr).
\eeq

\item The quadratic term $\mathcal{Q}^\epsilon_V$ is of the form
\beno
\mathcal{Q}^\epsilon_V=Q^\epsilon_{+,-}(V^+,V^-)+Q^\epsilon_{-,-}(V^-,V^-).
\eeno
And the symbol $q^\epsilon_{\mu,-}(\xi,\eta)$ of $Q^\epsilon_{\mu,-}$ satisfies
\beq\label{bound of q e}\begin{aligned}
&|q^\epsilon_{\mu,-}(\xi,\eta)|\lesssim\epsilon|\xi|\cdot\varphi_{\leq 5}\bigl(\sqrt\epsilon|\eta|\bigr)\cdot\varphi_{\leq -6}\Bigl(\f{|\xi-\eta|}{|\eta|}\Bigr).
\end{aligned}\eeq

\item The cubic term $\mathcal{L}^\epsilon_V=\epsilon^2\p_x(T_{B^\epsilon}V)$
satisfies
\beq\label{estimate for cubic term e 1}
\bigl|\text{Re}\bigl\{\bigl(\langle\p_x\rangle^{N_0}\mathcal{L}^\epsilon_V\,|\,\langle\p_x\rangle^{N_0}V\bigr)_2\bigr\}\bigr|
\lesssim\epsilon^2\|\z\|_{L^\infty}\|v\|_{H^2}\|V\|_{H^{N_0}}^2.
\eeq

\item The remaining nonlinear term $\mathcal{N}^\epsilon_V$ satisfies
\beq\label{estimate for nonlinear term e}
\|\mathcal{N}^\epsilon_V\|_{H^{N_0}}\lesssim\epsilon\bigl(\|\z\|_{W^{3,\infty}}+\|v\|_{W^{3,\infty}}\bigr)
\bigl(1+\|\z\|_{H^{N_0}}+\|v\|_{H^{N_0}}\bigr)^2
\bigl(\|\z\|_{H^{N_0}}+\|v\|_{H^{N_0}}\bigr).
\eeq
\end{itemize}
\end{proposition}
\begin{remark}
The terms $\mathcal{S}^\epsilon_V$, $\mathcal{Q}^\epsilon_V$ and $\mathcal{L}^\epsilon_V$ in \eqref{New formula for bsq e}  correspond
to $\mathcal{S}_V$, $\mathcal{Q}_V$ and $\mathcal{L}_V$ in \eqref{New formula} respectively. Whereas $\mathcal{N}^\epsilon_V$ in \eqref{New formula for bsq e} is corresponding to the sum  $\mathcal{R}_V+\mathcal{M}_V+\mathcal{C}_V+\mathcal{N}_V$ in \eqref{New formula}.
\end{remark}
\begin{remark}
Proposition \ref{New Bsq e prop} reveals that the worst term is $\mathcal{Q}^\epsilon_V$. Indeed, \eqref{bound of q e} hints that term $\mathcal{Q}^\epsilon_V$ is of order $O(\sqrt\epsilon)$ if there is no loss of derivative.
\end{remark}

\begin{proof}[Proof of Proposition \ref{New Bsq e prop}] Thanks to \eqref{New Bsq e}, rewriting \eqref{New Bsq e} to \eqref{New formula for bsq e}, we have
\beno\begin{aligned}
&\mathcal{S}^\epsilon_V=-\epsilon\p_x(T_vV)+\f{i}{2}\epsilon|\p_x|(T_\z V),\\
&\mathcal{Q}^\epsilon_V=-\f{\epsilon}{2}\p_x\bigl(T_\z \varphi_{\leq5}(\sqrt\epsilon|\p_x|)u\bigr)
-\epsilon\f{i}{2}|\p_x|\bigl(T_\z\varphi_{\leq5}(\sqrt\epsilon|\p_x|)\z\bigr),\\
&\mathcal{L}^\epsilon_V=\epsilon^2\p_x(T_{B^\epsilon}V),\\
&\mathcal{N}^\epsilon_V=\bigl(N_\z^\epsilon+\f{\epsilon}{2}\p_x\bigl(T_\z \varphi_{\leq5}(\sqrt\epsilon|\p_x|)u\bigr)\bigr)
+i\f{\p_x}{|\p_x|}\bigl(N_u^\epsilon-\f{\epsilon}{2}\p_x\bigl(T_\z\varphi_{\leq5}(\sqrt\epsilon|\p_x|)\z\bigr)\bigr).
\end{aligned}\eeno

 Thanks to \eqref{good variable for bsq e}, we have
\beq\label{Bsq 16}
\z=\f{1}{2}(V^++V^-)=\f12\sum_{\mu\in\{+,-\}}V^\mu,\quad u=\f{i}{2}\f{\p_x}{|\p_x|}(V^+-V^-)=\f{i}{2}\sum_{\mu\in\{+,-\}}\mu\f{\p_x}{|\p_x|}V^\mu.
\eeq

Using \eqref{Bsq 16}, we could rewrite $\mathcal{S}^\epsilon_V$ and $\mathcal{Q}^\epsilon_V$ in terms of $V^+$ and $V^-$. They would have similar expression as $\mathcal{S}_V$ and $\mathcal{Q}_V$ in the proof of Proposition \ref{New Bsq prop}. It is easy to check that there hold
\eqref{bound of s e} and \eqref{bound of q e}.

Similarly as in  the derivation of \eqref{estimate for cubic term 1}, using the symmetric structure of $\mathcal{L}^\epsilon_V$ and \eqref{estimate for B e 1}, we have
\beno
\bigl|\text{Re}\bigl\{\bigl(\langle\p_x\rangle^{N_0}\mathcal{L}^\epsilon_V\,|\,\langle\p_x\rangle^{N_0}V\bigr)_2\bigr\}\bigr|
\lesssim\epsilon^2\|B^\epsilon\|_{H^2}\|V\|_{H^{N_0}}^2,
\eeno
which along with \eqref{estimate for B e} implies the estimate \eqref{estimate for cubic term e 1}.

For the remained nonlinear term $\mathcal{N}^\epsilon_V$, similarly as in the derivation of the estimates involving $\mathcal{C}_V$ and $\mathcal{N}_V$ in the proof of Proposition \ref{New Bsq prop}, using product estimates and  \eqref{estimate for B e}, we obtain \eqref{estimate for nonlinear term e}. The proposition is proved.
\end{proof}

\subsection{Main a priori estimates for \eqref{Bsqeps}} Similarly as the proof of Theorem \ref{main theorem}, the proof of Theorem \ref{long time existence thm} also relies on the continuity argument and the a priori energy estimates. Before stating the main a priori energy estimates of \eqref{Bsqeps}, we present the ansatz for the continuity arguments.

The first ansatz is involving the amplitude of $\z$ as follows
\beq\label{ansatz 1 for bsq e}
\epsilon\|\z(t)\|_{L^\infty}\leq\f{1}{2C_{B^\epsilon}},\quad\text{for}\quad t\in[0,T_0\epsilon^{-\f23}].
\eeq

We define the energy functional for \eqref{Bsqeps} as
\beno
\mathcal{E}_{N_0}(t)=\|\z(t)\|_{H^{N_0}}^2+\|v(t)\|_{H^{N_0}}^2.
\eeno
For simplicity of the proof and without loss of generality, we assume
\beq\label{initial energy}
\|\z_0\|_{H^{N_0}}^2+\|v_0\|_{H^{N_0}}^2=1.
\eeq
Our second ansatz  is about the energy and reads
\beq\label{ansatz 2 for bsq e}
\mathcal{E}_{N_0}(t)\leq 2C_0', \quad\text{for}\quad t\in[0,T_0\epsilon^{-\f23}],
\eeq
where  $C_0'>1$ is an universal constant that will be determined in the end of the proof. We take
\beno
T_0=\f{C_1'}{C_2'},\quad C_0'=2C_1',
\eeno
where $C_1',C_2'$ are constants stated in the following Proposition \ref{estimate Prop for bsq e}. Thanks to Proposition \ref{estimate Prop for bsq e}, we could improve the ansatz \eqref{ansatz 1 for bsq e} and \eqref{ansatz 2 for bsq e}. Precisely, there exists a constant $\epsilon_0>0$ such that for any $\epsilon\in(0,\epsilon_0]$, we improve the ansatz \eqref{ansatz 1 for bsq e} and \eqref{ansatz 2 for bsq e} to
\beno\begin{aligned}
&\epsilon\|\z(t)\|_{L^\infty}\leq\f{1}{4C_{B^\epsilon}},\quad\text{for}\quad t\in[0,T_0\epsilon^{-\f23}]\\
\text{and}\quad&\mathcal{E}_{N_0}(t)\leq C_0', \quad\text{for}\quad t\in[0,T_0\epsilon^{-\f23}].
\end{aligned}\eeno
Then Theorem \ref{long time existence thm} follows from the above argument and the local regularity theorem.

Now, we focus on  the a priori energy estimate which is established in the following proposition.

\begin{proposition}\label{estimate Prop for bsq e}
Assume that $0<\epsilon<1$ and there holds \eqref{initial energy}. Under the ansatz \eqref{ansatz 1 for bsq e} and \eqref{ansatz 2 for bsq e}, the solution $(\z,v)$ of \eqref{Bsqeps}-\eqref{initialeps} satisfies
\beq\label{priori estimate for bsq e}
\mathcal{E}_{N_0}(t)\leq C_1'+C_2't\epsilon^{\f23},\quad\text{for any } t\in(0,T_0\epsilon^{-\f23}],
\eeq
where $C_1'$ and $C_2'$ are two universal constants, and $T_0=\f{C_1'}{C_2'}$.
\end{proposition}
\begin{proof} We shall use the formulation \eqref{New formula for bsq e} to derive the energy estimates for the Boussinesq system \eqref{Bsqeps}. Due to Proposition \ref{New Bsq e prop}, standard energy estimates will give rise to a local existence theorem with time scale of $O(1/{\sqrt\epsilon})$. To enlarge the existence time, we will apply the normal forms transformation to the worst term $\mathcal{Q}^\epsilon_V$. Now we sketch the proof.

{\bf Step 1. The a priori energy estimate.} Thanks to \eqref{estimate for B e} and \eqref{ansatz 1 for bsq e}, we have
\beno
\epsilon\|B^\epsilon(\z,v)\|_{H^{N_0}}\leq\f12\|v\|_{H^{N_0}},
\eeno
which along with \eqref{good unknowns for bsq e} and \eqref{good variable for bsq e} implies
\beq\label{equivalent energy functional for bsq e 1}
\mathcal{E}_{N_0}(t)\sim\|\z(t)\|_{H^{N_0}}^2+\|u(t)\|_{H^{N_0}}^2\sim\|V(t)\|_{H^{N_0}}^2,\quad\text{for}\quad t\in[0,T_0\epsilon^{-\f23}]
\eeq

By virtue of \eqref{equivalent energy functional for bsq e 1}, we start the energy estimate of \eqref{New formula for bsq e} as follows
\beno\begin{aligned}
&\f12\f{d}{dt}\|V(t)\|_{H^{N_0}}^2=\text{Re}\{\bigl(\langle\p_x\rangle^{N_0}\mathcal{S}^\epsilon_V\,|\,\langle\p_x\rangle^{N_0}V\bigr)_2
+\bigl(\langle\p_x\rangle^{N_0}\mathcal{Q}^\epsilon_V\,|\,\langle\p_x\rangle^{N_0}V\bigr)_2\\
&\qquad
+\bigl(\langle\p_x\rangle^{N_0}\mathcal{L}^\epsilon_V\,|\,\langle\p_x\rangle^{N_0}V\bigr)_2
+\bigl(\langle\p_x\rangle^{N_0}\mathcal{N}^\epsilon_V\,|\,\langle\p_x\rangle^{N_0}V\bigr)_2\}.
\end{aligned}\eeno
Thanks to the estimates \eqref{estimate for cubic term e 1} and \eqref{estimate for nonlinear term e} in Proposition \ref{New Bsq e prop},  using \eqref{initial energy}, \eqref{ansatz 2 for bsq e} and \eqref{equivalent energy functional for bsq e 1}, we obtain
\beq\label{priori for bsq e}
\mathcal{E}_{N_0}(t)\lesssim1+|\text{Re}(I)|+|\text{Re}(II)|+t\e,
\eeq
where
\beq\label{quadratic  terms e}\begin{aligned}
&I\eqdefa\sum_{\mu\in\{+,-\}}\int_0^t\int_{\R^2}\langle\xi\rangle^{2N_0}s^\epsilon_{\mu,+}(\xi,\eta)\widehat{V^\mu}(\xi-\eta)\widehat{V^+}(\eta)
\overline{\widehat{V^+}(\xi)}d\eta d\xi dt,\\
&II\eqdefa\sum_{\mu\in\{+,-\}}\int_0^t\int_{\R^2}\langle\xi\rangle^{2N_0}q^\epsilon_{\mu,-}(\xi,\eta)\widehat{V^\mu}(\xi-\eta)\widehat{V^-}(\eta)
\overline{\widehat{V^+}(\xi)}d\eta d\xi dt.
\end{aligned}\eeq

\medskip

{\bf Step 2. Estimate for Re$(I)$.} Similarly as Step 3 in the proof of Proposition \ref{New Bsq prop}, using symmetric structure of $\mathcal{S}^\epsilon_V$, we have
\beno
\text{Re}(I)=\f12(I+\bar{I})=\sum_{\mu\in\{+,-\}}\int_0^t\int_{\R^2}\tilde{s}^\epsilon_{\mu,+}(\xi,\eta)\widehat{V^{\mu}}(\xi-\eta)\cdot
\langle\eta\rangle^{N_0}\widehat{V^+}(\eta)\cdot\langle\xi\rangle^{N_0}\widehat{V^-}(-\xi)d\eta d\xi dt,
\eeno
where
\beno\begin{aligned}
\tilde{s}^\epsilon_{\mu,+}(\xi,\eta)&=\langle\xi\rangle^{-N_0}\langle\eta\rangle^{-N_0}\bigl(\langle\xi\rangle^{2N_0}s^\epsilon_{\mu,+}(\xi,\eta)
-\langle\eta\rangle^{2N_0}s^\epsilon_{-\mu,+}(\eta,\xi)\bigr).
\end{aligned}\eeno
Thanks to \eqref{bound of s e}, we have
\beno
|\tilde{s}^\epsilon_{\mu,+}(\xi,\eta)|\lesssim\epsilon|\xi-\eta|\cdot\varphi_{\leq -6}\Bigl(\f{|\xi-\eta|}{\max\{|\xi|,|\eta|\}}\Bigr).
\eeno
Then we obtain
\beno\begin{aligned}
&|\text{Re}(I)|\lesssim\epsilon\int_0^t\int_{\R^2}|\xi-\eta||\widehat{V}(\xi-\eta)|\cdot
\langle\eta\rangle^{N_0}|\widehat{V}(\eta)|\cdot\langle\xi\rangle^{N_0}|\widehat{V}(\xi)|d\eta d\xi dt\\
&\lesssim\epsilon t\sup_{(0,t)}\|\xi\widehat{V}(\xi)\|_{L^1}\cdot
\|\langle\xi\rangle^{N_0}\widehat{V}(\xi)\|_{L^2}^2\lesssim\epsilon t\sup_{(0,t)}\|V\|_{H^2}\|V\|_{H^{N_0}}^2
\end{aligned}\eeno
which along with \eqref{initial energy}, \eqref{ansatz 2 for bsq e} and \eqref{equivalent energy functional for bsq e 1} implies
\beq\label{estimate for I e}
|\text{Re}(I)|\lesssim\epsilon t.
\eeq

\medskip

{\bf Step 3. Estimate for Re$(II)$.} Due to \eqref{bound of q e}, the direct estimate for $II$ will lead to one derivative loss or $\sqrt\epsilon$ loss. To improve the estimate, we shall apply the normal form transformation to this term.

{\it Step 4.1. The evolution equation and estimates of the profile}. Firstly, we  introduce the profiles $f$, $g$ of $V$ and $\langle\p_x\rangle^{N_0}V$ as follows
\beno
f=e^{-it\Lambda_\epsilon}V\quad\text{and}\quad g=\langle\p_x\rangle^{N_0}f.
\eeno
Thanks to \eqref{equivalent energy functional for bsq e 1}, we have
\beq\label{equivalent energy functional for bsq e 2}
\mathcal{E}_{N_0}(t)\sim\|V\|_{H^{N_0}}^2\sim\|f\|_{H^{N_0}}^2=\|g\|_{L^2}^2.
\eeq
Due to the equation \eqref{New Bsq e}, we have
\beq\label{evolution equation for profile of bsq e}
\p_tf=e^{-it\Lambda_\epsilon}\Bigl(-\epsilon\p_x(T_vV)+\f{i}{2}\epsilon|\p_x|(T_{\z}V)+N^\epsilon_\z+i\f{\p_x}{|\p_x|}N^\epsilon_u\Bigr).
\eeq

By virtue of definition of $B^\epsilon(\cdot,\cdot)$,  we have
\beno
\text{supp}\, \widehat{B^\epsilon(\cdot,\cdot)}(\xi)\subset\{\xi\in\R\,|\, \sqrt\epsilon|\xi|\geq 2^5\},
\eeno
which along with the expressions of $N^\epsilon_\z$ and $N^\epsilon_u$ implies
\beno\begin{aligned}
&\varphi_{\leq 0}(\sqrt\epsilon|\p_x|) N^\epsilon_\z=-\epsilon\p_x \varphi_{\leq 0}(\sqrt\epsilon|\p_x|)\bigl(\f12T_\z \varphi_{\leq 5}(\sqrt\epsilon|\p_x|)v+R(\z,v)\bigr),\\
&
\quad \varphi_{\leq 0}(\sqrt\epsilon|\p_x|) N^\epsilon_u=\f{\epsilon}{2}\p_x \varphi_{\leq 0}(\sqrt\epsilon|\p_x|)\bigl(T_\z \varphi_{\leq 5}(\sqrt\epsilon|\p_x|)\z-R(v,v)\bigr).
\end{aligned}\eeno
Then we have
\beno
\|\f{1}{|\p_x|}\varphi_{\leq 0}(\sqrt\epsilon|\p_x|) \p_tf\|_{H^{N_0}}\lesssim\epsilon\bigl(\|\z\|_{L^\infty}+\|v\|_{L^\infty}\bigr)\bigl(\|\z\|_{H^{N_0}}+\|v\|_{H^{N_0}}+\|V\|_{H^{N_0}}\bigr).
\eeno
Due to the expressions of $N^\epsilon_\z$ and $N^\epsilon_u$, using \eqref{estimate for B e}, we also have
\beno\begin{aligned}
&\|\f{1}{|\p_x|}\varphi_{\geq 1}(\sqrt\epsilon|\p_x|)\p_tf\|_{H^{N_0}}\lesssim\epsilon
\bigl(\|v\|_{W^{3,\infty}}+\|\z\|_{W^{3,\infty}}\bigr)\bigl(1+\|v\|_{W^{3,\infty}}+\|\z\|_{W^{3,\infty}}\bigr)\\
&\qquad\times
\bigl(\|V\|_{H^{N_0}}+\|v\|_{H^{N_0}}+\|\z\|_{H^{N_0}}+\|u\|_{H^{N_0}}\bigr)
\end{aligned}\eeno

Thanks to  \eqref{initial energy}, \eqref{ansatz 2 for bsq e} and \eqref{equivalent energy functional for bsq e 1}, we obtain
\beq\label{estimate for p_t f e}
\|\f{1}{|\p_x|}\p_tf\|_{H^{N_0}}\lesssim\epsilon.
\eeq

{\it Step 4.2. The profiles version for $II$. } Denoting by
\beno
\mathfrak{Q}^\epsilon_{\mu}(\xi,\eta)=\langle\xi\rangle^{2N_0}q^\epsilon_{\mu,-}(\xi,\eta)\widehat{V^\mu}(\xi-\eta)\widehat{V^-}(\eta)
\overline{\widehat{V^+}(\xi)},
\eeno
we have
\beno
II=\int_0^t\int_{\R^2}\mathfrak{Q}^\epsilon_{+}(\xi,\eta)d\eta d\xi dt+\int_0^t\int_{\R^2}\mathfrak{Q}^\epsilon_{-}(\xi,\eta)d\eta d\xi dt.
\eeno

Now we rewrite $\mathfrak{Q}^\epsilon_{\mu}(\xi,\eta)$ in terms of the profiles $f$ and $g$ as follows
\beno
\mathfrak{Q}^\epsilon_{\mu}(\xi,\eta)=e^{it\Phi^\epsilon_{\mu,-}(\xi,\eta)}\tilde{q}^\epsilon_{\mu,-}(\xi,\eta)\widehat{f^{\mu}}(\xi-\eta)\cdot
\widehat{g^-}(\eta)\cdot\widehat{g^-}(-\xi),
\eeno
where
\beno\begin{aligned}
&\Phi^\epsilon_{\mu,-}(\xi,\eta)=-\Lambda_\epsilon(\xi)+\mu\Lambda_\epsilon(\xi-\eta)-\Lambda_\epsilon(\eta),\\
&\tilde{q}^\epsilon_{\mu,-}(\xi,\eta)=\langle\eta\rangle^{-N_0} \langle\xi\rangle^{N_0} q^\epsilon_{\mu,-}(\xi,\eta).
\end{aligned}\eeno
Thanks to \eqref{bound of q e}, we have
\beq\label{bound of tilde q e}\begin{aligned}
&|\tilde{q}^\epsilon_{\mu,-}(\xi,\eta)|\lesssim\epsilon|\xi|\cdot\varphi_{\leq 5}\bigl(\sqrt\epsilon|\eta|\bigr)\cdot\varphi_{\leq -6}\Bigl(\f{|\xi-\eta|}{|\eta|}\Bigr),\\
&\text{supp}\, \tilde{q}^\epsilon_{\mu,-}\subset\mathbb{S}^\epsilon
\eqdefa\{(\xi,\eta)\in\R^2\,|\,\xi\cdot\eta>0,\,\,\f{31}{32}|\eta|\leq|\xi|\leq\f{33}{32}|\eta|,\,\, \sqrt\epsilon|\eta|\leq 2^6\}.
\end{aligned}\eeq

Lemma \ref{phase lem for bsq e} and the fact $\xi\cdot\eta>0$ (in \eqref{bound of tilde q e}) yield
\beno
\Phi^\epsilon_{+,-}(\xi,\eta)=\f12\min\{|\xi|,|\eta|\}\phi^\epsilon_{+,-}(\xi,\eta),
\eeno
with
\beno
\phi^\epsilon_{+,-}(\xi,\eta)=
\left\{\begin{aligned}
&6\epsilon\xi^2-6\epsilon\xi\cdot\eta+4\epsilon\eta^2-4,\quad\text{if}\quad |\xi|>|\eta|,\\
&6\epsilon\eta^2-6\epsilon\xi\cdot\eta+4\epsilon\xi^2-4,\quad\text{if}\quad |\xi|<|\eta|.
\end{aligned}\right.
\eeno
Then there hold
\beq\label{phase relation}
\Phi^\epsilon_{-,-}(\xi,\eta)=\Phi^\epsilon_{+,-}(\eta,\xi)\quad\text{and}\quad \Phi^\epsilon_{+,-}(\xi,\eta)=\Phi^\epsilon_{+,-}(\eta,\xi).
\eeq
With \eqref{phase relation}, we only derive the estimate for the integral of $\mathfrak{Q}^\epsilon_{+}(\xi,\eta)$ over set $\mathbb{S}^\epsilon_{>}$ with
\beno
\mathbb{S}^\epsilon_{>}=\{(\xi,\eta)\in\mathbb{S}^\epsilon\,|\,|\xi|>|\eta|\}.
\eeno

{\it Step 4.3. Estimate for $\int_0^t\int_{\mathbb{S}^\epsilon_{>}}\mathfrak{Q}^\epsilon_{+}(\xi,\eta)d\eta d\xi dt$.}
We divide $\mathfrak{Q}^\epsilon_{+}(\xi,\eta)$ into three parts as follows:

{\it (1). For low frequency  $\sqrt\epsilon|\eta|\leq\f12$},  using \eqref{bound of tilde q e}, we have
\beq\label{I 21}
|\phi^\epsilon_{+,-}(\xi,\eta)|\sim 1\quad\text{and}\quad
|\f{\tilde{q}^\epsilon_{+,-}(\xi,\eta)}{i\Phi^\epsilon_{+,-}(\xi,\eta)}|\lesssim\f{\epsilon}{|\phi^\epsilon_{+,-}(\xi,\eta)|}\lesssim \epsilon.
\eeq
Integrating by parts w.r.t t, we have
\beno\begin{aligned}
&\int_0^t\int_{\mathbb{S}^\epsilon_{>}}\mathfrak{Q}^\epsilon_{+}(\xi,\eta)\varphi_{\leq-2}(\sqrt\epsilon|\eta|)d\eta d\xi dt\\
&=\underbrace{\int_{\mathbb{S}^\epsilon_{>}}\f{\tilde{q}^\epsilon_{+,-}(\xi,\eta)}{i\Phi^\epsilon_{+,-}(\xi,\eta)}
e^{it\Phi^\epsilon_{+,-}(\xi,\eta)}\widehat{f^+}(\tau,\xi-\eta)\cdot
\widehat{g^-}(\tau,\eta)\cdot\widehat{g^-}(\tau,-\xi)\varphi_{\leq-2}(\sqrt\epsilon|\eta|)d\eta d\xi}_{A_1^\epsilon}|_{\tau=0}^t\\
&\quad-\underbrace{\int_0^t\int_{\mathbb{S}^\epsilon_{>}}\f{\tilde{q}^\epsilon_{+,-}(\xi,\eta)}{i\Phi^\epsilon_{+,-}(\xi,\eta)}
e^{it\Phi^\epsilon_{+,-}(\xi,\eta)}\p_t\bigl(\widehat{f^+}(\xi-\eta)\cdot
\widehat{g^-}(\eta)\cdot\widehat{g^-}(-\xi)\bigr)\varphi_{\leq-2}(\sqrt\epsilon|\eta|)d\eta d\xi dt}_{A_2^\epsilon}
\end{aligned}\eeno

Similarly as the derivation of \eqref{estimate for I>0 in high fre} in Step 3.1 of proof to Proposition \ref{New Bsq prop}, using \eqref{bound of tilde q e} and \eqref{I 21}, we have
\beno\begin{aligned}
&|A_1^\epsilon|\lesssim\epsilon\int_{\mathbb{S}^\epsilon_{>}}|\widehat{f}(\xi-\eta)|\cdot
|\widehat{g}(\eta)|\cdot|\widehat{g}(-\xi)|d\eta d\xi\lesssim\epsilon\|\widehat{f}(\xi)\|_{L^1}\|\widehat{g}(\xi)\|_{L^2}^2\lesssim\epsilon\|f\|_{H^1}\|g\|_{L^2}^2,\\
& |A_2^\epsilon|\lesssim\epsilon t\sup_{(0,t)}\int_{\mathbb{S}^\epsilon_{>}}\bigl(|\p_t\widehat{f}(\xi-\eta)|\cdot
|\widehat{g}(\eta)|\cdot|\widehat{g}(-\xi)|+|\widehat{f}(\xi-\eta)|\cdot
|\p_t\bigl(\widehat{g^-}(\eta)\cdot\widehat{g^-}(-\xi)\bigr)|\bigr)\varphi_{\leq-2}(\sqrt\epsilon|\eta|)d\eta d\xi\\
&\lesssim\epsilon t\sup_{(0,t)}\bigl(\|\p_tf\|_{H^1}\|g\|_{L^2}^2+\f{1}{\sqrt\epsilon}\|f\|_{H^1}\|\f{1}{|\p_x|}\p_tg\|_{L^2}\|g\|_{L^2}\bigr),
\end{aligned}\eeno
where we used the fact that $|\xi|\sim|\eta|$ and the following inequality in the last inequality
\beno
\sqrt\epsilon|\eta|\varphi_{\leq-2}(\sqrt\epsilon|\eta|)\lesssim1.
\eeno
Thanks to  \eqref{initial energy}, \eqref{ansatz 2 for bsq e}, \eqref{equivalent energy functional for bsq e 2} and \eqref{estimate for p_t f e}, we obtain
\beq\label{estimate for II e in low fre}
|\int_0^t\int_{\mathbb{S}^\epsilon_{>}}\mathfrak{Q}^\epsilon_{+}(\xi,\eta)\varphi_{\leq-2}(\sqrt\epsilon|\eta|)d\eta d\xi dt|
\lesssim\epsilon+\epsilon^{\f32}t.
\eeq

{\it (2). For moderate frequencies with large modulation of phase}, i.e., for
\beno
\f14\leq\sqrt\epsilon|\eta|\leq 2^6\quad\text{and}\quad |\phi^\epsilon_{+,-}(\xi,\eta)|\geq 2^{-D-1},
\eeno
we have
\beno
|\f{\tilde{q}^\epsilon_{+,-}(\xi,\eta)}{i\Phi^\epsilon_{+,-}(\xi,\eta)}|\lesssim\f{\epsilon}{|\phi^\epsilon_{+,-}(\xi,\eta)|}\lesssim 2^D\epsilon.
\eeno
Following similar arguments as \eqref{estimate for II e in low fre}, integrating by parts with respect to  t, we get
\beq\label{estimate for II e in med fre with large modulation phase}
|\int_0^t\int_{\mathbb{S}^\epsilon_{>}}\mathfrak{Q}^\epsilon_{+}(\xi,\eta)
\varphi_{[-1,5]}(\sqrt\epsilon|\eta|)\varphi_{\geq-D}(\phi_{+,-}(\xi,\eta))d\eta d\xi dt|
\lesssim2^D\epsilon+2^D\epsilon^{\f32}t.
\eeq

{\it (3). For moderate frequencies with small modulation of phase}, i.e., for
\beno
\f14\leq\sqrt\epsilon|\eta|\leq 2^6\quad\text{and}\quad |\phi^\epsilon_{+,-}(\xi,\eta)|\leq 2^{-D},
\eeno
we divide the integral set into the following two parts
\beno
\underbrace{\{(\xi,\eta)\in \mathbb{S}^\epsilon\,|\, 0<\eta<\xi\leq\f{33}{32}\eta,\,\,\f14\leq\sqrt\epsilon\eta\leq 2^6\}}_{\mathbb{S}^\epsilon_{>,+}}\cup
\underbrace{\{(\xi,\eta)\in \mathbb{S}^\epsilon\,|\, 0>\eta>\xi\geq\f{33}{32}\eta,\,\,-\f14\geq\sqrt\epsilon\eta\geq -2^6\}}_{\mathbb{S}^\epsilon_{>,-}}.
\eeno
We only derive the estimate for the integral over the set $\mathbb{S}^\epsilon_{>,+}$. Now, introducing the coordinates transformation on $\mathbb{S}^\epsilon_{>,+}$ as follows:
\beno\begin{aligned}
\Psi_\epsilon:\,&\mathbb{S}^\epsilon_{>,+}\rightarrow\widetilde{\mathbb{S}^\epsilon}_{>,+}\subset\R^2,\\
&(\xi,\eta)\mapsto(\tilde\xi,\eta)=(\phi^\epsilon_{+,-}(\xi,\eta),\eta),
\end{aligned}\eeno
we have
\beq\label{det}
\det\,\Bigl(\f{\p\Psi_\epsilon(\xi,\eta)}{\p(\xi,\eta)}\Bigr)=\f{\p\phi^\epsilon_{+,-}(\xi,\eta)}{\p\xi}=\epsilon(12\xi-6\eta)
\sim\epsilon\eta\sim\sqrt\epsilon.
\eeq
Then $\Psi_\epsilon$ is invertible and we denote by
\beno
(\xi,\eta)=\Psi_\epsilon^{-1}(\tilde\xi,\eta).
\eeno

Changing the variables $(\xi,\eta)$ to $(\tilde\xi,\eta)$, using \eqref{bound of tilde q e} and \eqref{det}, we have
\beno\begin{aligned}
&|\int_0^t\int_{\mathbb{S}^\epsilon_{>,+}}\mathfrak{Q}^\epsilon_{+}(\xi,\eta)
\varphi_{[-1,5]}(\sqrt\epsilon|\eta|)\varphi_{\leq-D-1}(\phi_{+,-}(\xi,\eta))d\eta d\xi dt|\\
&\lesssim t\sup_{(0,t)}\int_{\f{1}{4\sqrt\epsilon}}^{\f{32}{\sqrt\epsilon}}\int_{-2^{-D}}^{2^{-D}}
\bigl(\sqrt\epsilon|\xi||\widehat{f}(\xi-\eta)|\cdot
|\widehat{g}(\eta)|\cdot|\widehat{g}(\xi)|1_{\mathbb{S}^\epsilon_{>,+}}\bigr)
|_{(\xi,\eta)=\Psi_\epsilon^{-1}(\tilde\xi,\eta)}d\tilde\xi d\eta\\
&\lesssim t2^{-\f{D}{2}}\sup_{(0,t)}\|g\|_{L^2}\Bigl(\int_{\f{1}{4\sqrt\epsilon}}^{\f{32}{\sqrt\epsilon}}\int_{-2^{-D}}^{2^{-D}}
\bigl(|\widehat{f}(\xi-\eta)|^2\cdot|\widehat{g}(\xi)|^21_{\mathbb{S}^\epsilon_{>,+}}\bigr)
|_{(\xi,\eta)=\Psi_\epsilon^{-1}(\tilde\xi,\eta)}d\tilde\xi d\eta\Bigr)^{\f12},
\end{aligned}\eeno
where we used the fact that $\sqrt\epsilon|\xi|\sim\sqrt\epsilon|\eta|\sim1$ in the last inequality. Then changing variables $(\tilde\xi,\eta)$
to $(\xi,\eta)$, using \eqref{det}, we have
\beno
|\int_0^t\int_{\mathbb{S}^\epsilon_{>,+}}\mathfrak{Q}^\epsilon_{+}(\xi,\eta)
\varphi_{[-1,5]}(\sqrt\epsilon|\eta|)\varphi_{\leq-D-1}(\phi_{+,-}(\xi,\eta))d\eta d\xi dt|
\lesssim 2^{-\f{D}{2}}\epsilon^{\f14}t\sup_{(0,t)}\|f\|_{L^2}\|g\|_{L^2}^2,
\eeno
which along with \eqref{initial energy}, \eqref{ansatz 2 for bsq e} and \eqref{equivalent energy functional for bsq e 2} implies
\beq\label{estimate for II e in med fre with small modulation phase}
|\int_0^t\int_{\mathbb{S}^\epsilon_{>,+}}\mathfrak{Q}^\epsilon_{+}(\xi,\eta)
\varphi_{[-1,5]}(\sqrt\epsilon|\eta|)\varphi_{\leq-D-1}(\phi_{+,-}(\xi,\eta))d\eta d\xi dt|
\lesssim2^{-\f{D}{2}}\epsilon^{\f14}t.
\eeq
The same estimate holds for the integral over set $\mathbb{S}^\epsilon_{>,-}$.

Taking $D=[\log_2\epsilon^{-\f56}]$ (i.e., $2^D\sim\epsilon^{-\f56}$) in \eqref{estimate for II e in med fre with large modulation phase}
and \eqref{estimate for II e in med fre with small modulation phase}, together with \eqref{estimate for II e in low fre}, we obtain that
\beq\label{estimate for II e>}
|\int_0^t\int_{\mathbb{S}^\epsilon_{>}}\mathfrak{Q}^\epsilon_{+}(\xi,\eta)d\eta d\xi dt|
\lesssim1+\epsilon^{\f23}t.
\eeq
The same estimates hold for $\int_0^t\int_{\R^2}\mathfrak{Q}^\epsilon_{+}(\xi,\eta)d\eta d\xi dt$ and $\int_0^t\int_{\R^2}\mathfrak{Q}^\epsilon_{-}(\xi,\eta)d\eta d\xi dt$. Then we  obtain
\beq\label{estimate for II e}
|\text{Re}(II)|
\lesssim1+\epsilon^{\f23}t.
\eeq

\medskip

{\bf Step 5. Final energy estimates.} Combining \eqref{priori for bsq e}, \eqref{estimate for I e} and \eqref{estimate for II e}, we finally obtain
\beno
\mathcal{E}_{N_0}(t)\lesssim 1+\epsilon^{\f23}t.
\eeno
This is exactly \eqref{priori estimate for bsq e}. This completes the proof of the proposition.
\end{proof}

\medskip

\section{Final comments}


1. It would be interesting to extend the results of the present paper to the two-dimensional version of \eqref{Bsq} or \eqref{Bsqeps}.

2. As for other Boussinesq systems except those described in Remark 1.1, the global well-posedness (or finite time blow-up) of \eqref{Bsqeps} is an open question.

\vspace{0.5cm}
\noindent {\bf Acknowledgments.}  The work
of the second author was partially supported by NSF of China under grants 11671383 and
by an innovation grant from National Center for Mathematics and Interdisciplinary Sciences.

\end{document}